\documentclass[11pt]{amsart}
  
\usepackage{hyperref}
\usepackage{latexsym}
\usepackage{graphicx}
\usepackage[normalem]{ulem}
\usepackage{amssymb}
\usepackage{epstopdf}

\usepackage[show]{ed}

\newcommand*{\rom}[1]{\expandafter\@slowromancap\romannumeral #1@}

\renewcommand{\Im}{{\operatorname{Im}\,}}

\newcommand{\sgn}{\operatorname{sgn}}

\renewcommand{\epsilon}{\varepsilon}

\newcommand{\mc}[1]{\mathcal{#1}}
\newcommand{\e}{\varepsilon}

\newcommand{\re}{\mathbb{R}}

\newcommand{\supp}{{\operatorname{supp\,}}}

\newtheorem{theorem}{{\sc Theorem}}

\newtheorem{cor}{{\sc Corollary}}[section]

\newtheorem{definition}[cor]{{\sc Definition}}
\newtheorem{lemma}[cor]{{\sc Lemma}}

\newtheorem{conjecture}[cor]{{\sc Conjecture}}
\numberwithin{equation}{section}

\newenvironment{rem}[1][]{\refstepcounter{cor}{\medskip\noindent{{\sc{Remark}}{\bf~\thecor \/#1.}} }}{\par}

\title{Domains without dense Steklov nodal sets}

\author{Oscar Bruno}
\address{Department of Computing and Mathematical Sciences, Caltech, Pasadena, CA USA}
\email{obruno@caltech.edu}
\author{Jeffrey Galkowski}
\address{Department of Mathematics, Northeastern, Boston, MA USA}
\email{jeffrey.galkowski@northeastern.edu }

\date{}

\graphicspath{{figures/}}
\begin{document}
\begin{abstract}
  This {article} concerns the asymptotic geometric character of the
  nodal set of the eigenfunctions of the Steklov eigenvalue problem
$$
-\Delta \phi_{\sigma_j}=0,\quad\text{ on
}\Omega,\qquad\qquad \partial_\nu \phi_{\sigma_j}=\sigma_j
\phi_{\sigma_j}\quad \text{ on }\partial\Omega
$$
in two-dimensional domains $\Omega$. In particular, this paper
presents a dense family $\mc{A}$ of simply-connected two-dimensional
domains with analytic boundaries such that, for each
$\Omega\in \mc{A}$, the nodal set of the eigenfunction
$\phi_{\sigma_j}$ ``is \emph{not} dense at scale
$\sigma_j^{-1}$''. This result addresses a question put forth under
``Open Problem 10'' in Girouard and Polterovich, J.  Spectr. Theory,
321-359 (2017).  In fact{,} the results in the present paper establish
that, for domains $\Omega\in \mc{A}$, the nodal sets of the
eigenfunctions $\phi_{\sigma_j}$ associated with the eigenvalue
$\sigma_j$ have {starkly} different character than anticipated: they
are not dense at any {shrinking} scale. More precisely, for each
$\Omega\in \mc{A}$ there is a value $r_1>0$ such that {for each} $j$
{there is $x_j\in \Omega$} such that $\phi_{\sigma_j}$ does not vanish
on {the ball of radius $r_1$ around $x_j$}.
\end{abstract}

\maketitle

\section{Introduction}

Let $(M,g)$ be a compact Riemannian manifold with piecewise smooth
boundary $\partial M$. The Steklov problem is given by
\begin{equation}
\label{e:steklovM}
\begin{cases}
-\Delta_g \phi_{\sigma}=0&\text{in }M\\
\partial_\nu \phi_{\sigma}=\sigma \phi_{\sigma}&\text{on }\partial M.
\end{cases}
\end{equation}
There is a discrete sequence
$0=\sigma_0<\sigma_1\leq \sigma_2\leq \dots$ of values of $\sigma$,
with $\sigma_j\to \infty$ as $j\to \infty$, for which non-trivial
solutions satisfying~\eqref{e:steklovM} exist~\cite{HiLu}. These are
the \emph{Steklov eigenvalues} and the corresponding functions
$\phi_{\sigma_j}$ are the \emph{Steklov eigenfunctions}. This paper
studies the asymptotic character of the nodal set of the
eigenfunctions of the Steklov eigenvalue problem in the case $M$
equals a bounded open set $\Omega\in \re^2$. In particular the results
in this paper show that the nodal set of the eigenfunction
$\phi_{\sigma_j}$ is \emph{not} dense at scale $\sigma_j^{-1}$ for
some such sets $\Omega$---or, more precisely, that there is a dense
family $\mc{A}$ of simply-connected two-dimensional domains with
analytic boundaries such that, for each $\Omega\in \mc{A}$, the
eigenfunction $\phi_{\sigma_j}$ in the domain $\Omega$ remains nonzero
on a $j$-dependent ball of {$j$-independent} radius. This result
addresses a question put forth under ``Open Problem 10''
in~\cite{GiPo17}.

The behavior of both the Steklov eigenvalues (see
e.g.~\cite{GiPo17,GiPaPoSh,LePaPoSh}) and eigenfunctions (see
e.g.~\cite{PoShTo,GalkToth,BeLi,Zh16,Ze15,SoWaZh,Sha,HiLu}) have been
a topic of recent interest. When $M$ has smooth boundary, the Steklov
eigenfunctions $\phi_{\sigma_j}|_{\partial M}$ behave much like high
energy Laplace eigenfunctions with eigenvalue $\sigma_j^2$. In
particular, they oscillate at frequency
$\sigma_j$. References~\cite{PoShTo,BeLi,Zh16,Ze15,SoWaZh,WaZh,GeRF,Zh15}
study the nodal sets of $\phi_{\sigma_j}|_M$, giving both upper and
lower bounds on its Hausdorff measure similar to those for Laplace
eigenfunctions. In fact, most results regarding Steklov eigenfunctions
in the interior of $M$ extract behavior similar to that of high energy
Laplace eigenfunctions.

The purpose of this article is to show that, away from the boundary of
$M$, Steklov eigenfunctions behave \emph{very} differently than high
energy Laplace eigenfunctions. Not only do they decay rapidly
(see~\cite{GalkToth,HiLu}) but, at least for a dense class of
analytic domains, they oscillate slowly over certain portions of the
domain. Girouard--Polterovich~\cite[Open Problem 10(i)]{GiPo17} raise
the question of whether nodal sets of Steklov eigenfunctions are dense
at scale $\sigma_j^{-1}$ in $M$. One consequence of the results in the
present paper is a negative answer to this question. We show that
arbitrarily close to any simply-connected domain with analytic
boundary $\Omega_0\subset \re^2$, there is a domain $\Omega_1$ for
which the nodal sets are \emph{not} $\sigma_j^{-1}$ dense and, indeed,
that there is a region within $\Omega_1$ where the nodal set density
does not increase as $\sigma_j\to \infty$.  Moreover, the Steklov
eigenfunctions oscillate no faster than a fixed frequency in this
region. These results are summarized in the following theorem.
\begin{theorem}
\label{t:nodal}
Let $\Omega_0\subset \re^2$ be a bounded simply-connected domain with
analytic boundary, and let $k>0$ and $\e>0$ be given. Then there exist
a set $\Omega_1\subset \re^2$ with analytic boundary given by
\begin{equation}\label{e:pert_dom}
\partial\Omega_1=\{ x+{\nu} g(x)\mid x\in \partial\Omega_0\},\qquad\qquad \|g\|_{C^k(\partial\Omega_0)}<\e
\end{equation}
(where ${\nu}$ denotes the outward unit normal to
$\partial\Omega_0$ and where $g$ is an analytic function defined on
$\partial\Omega_0$), a point $x_0\in\Omega_1$ and numbers $0<r_1<r_0$,
 ($B(x_0, r_0)\subset \Omega_1$) such that: for each Steklov
eigenvalue $\sigma$ for the domain $\Omega_1$ there exists a point
$x_\sigma\in B(x_0,r_0)$ such that $B(x_\sigma, r_1)\subset B(x_0, r_0)$
and each Steklov eigenfunction $\phi_\sigma$ of eigenvalue $\sigma$
for the domain $\Omega_1$ satisfies
$$
|\phi_{\sigma}|>0\text{ on }B(x_\sigma, r_1)\subset \Omega_1.
$$
Additionally, ``$\phi_\sigma$ has bounded frequency on $B(x_0,r_0)$''
(a precise statement follows in Theorem~\ref{thm:main}).
\end{theorem}

\begin{figure}[!h]
  \centering
  \includegraphics[width=.45\textwidth]{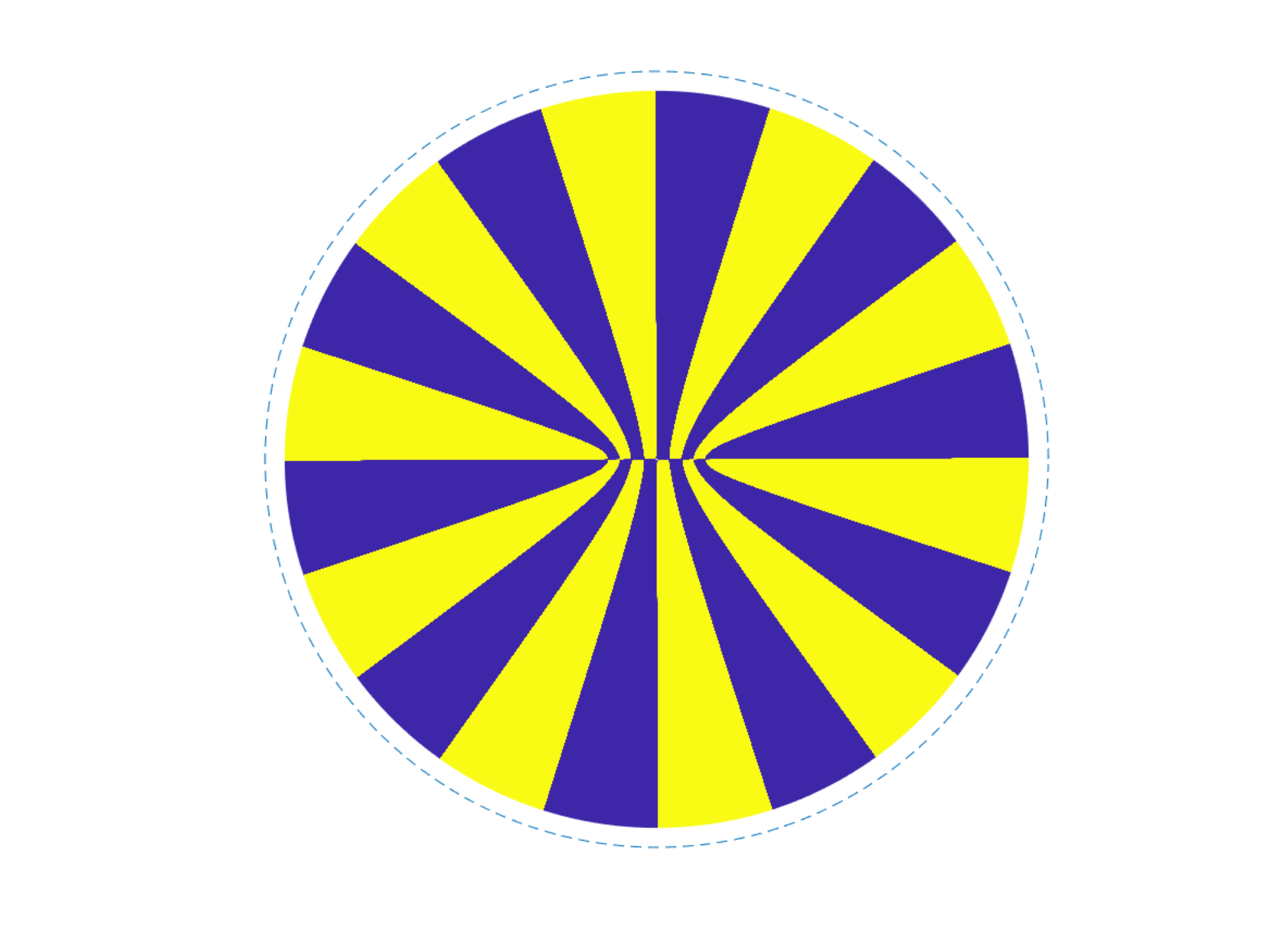}
   \includegraphics[width=.45\textwidth]{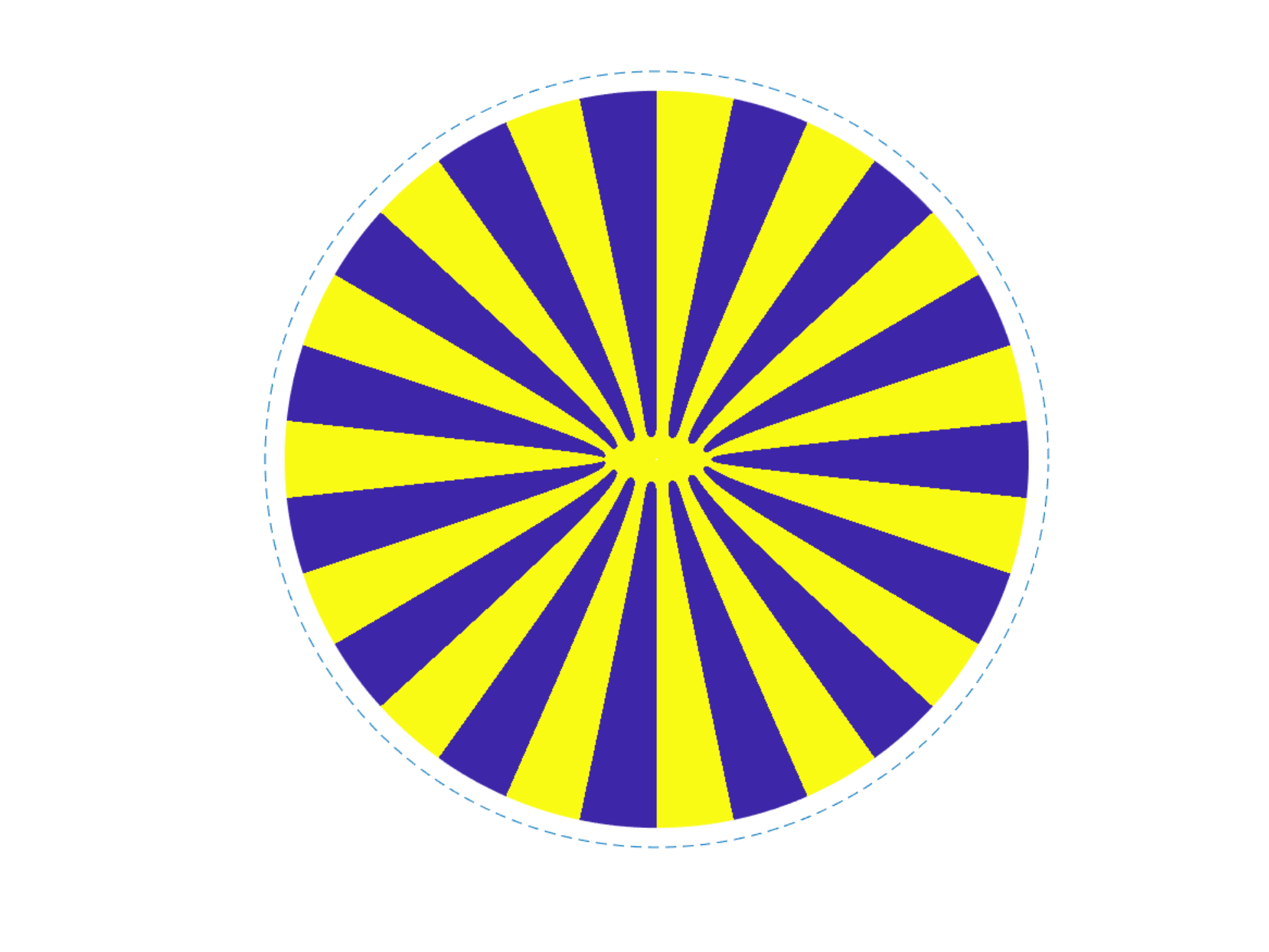}
   \caption{Fixed-sign sets for Steklov eigenfunctions over the
     elliptical domain $\Omega = x^2+\frac{y^2}{1.01^2}=1$. The yellow
     and blue regions indicate the subsets over which the
     eigenfunctions are positive and negative, respectively. The left
     and right images correspond to the eigenvalues
     $\sigma_{20}=9.9502$ and $\sigma_{30}=14.9253$, respectively. For
     a circle the nodal lines coincide with a set of $j$ uniformly
     arranged radial lines from the center to the boundary: they are
     dense at scale $\sigma_j^{-1}=j^{-1}$ over the complete domain,
     including the origin. Under the barely-visible perturbation of
     the unit disc into the slightly elliptical domain $\Omega$,
     regions of asymptotically fixed size on which the eigenfunction
     does not change sign open-up within $\Omega$. Indeed, the nodal
     set corresponding to $\sigma_{30}$ (right image) shows such an
     opening, whereas the nodal set corresponding to $\sigma_{20}$
     (left image) does not; cf. also
     Remark~\ref{disclaimer}. \label{f:EllipseEigenvalue}}
\end{figure}

Theorem~\ref{t:nodal} is a consequence of {the more} precise results
presented in Theorems~\ref{thm:main} and~\ref{t:tunnel} and
Corollary~\ref{c:dense}. In particular, these results establish that,
for each domain $\Omega$ in a dense class $\mathcal A$ of
two-dimensional domains, an estimate holds for the truncation error in
certain ``mapped Fourier expansions'' of the eigenfunctions
$\phi_\sigma$ (i.e., Fourier expansions of $\phi_\sigma$ under a
change of variables). {This estimate} is uniformly valid {over a
  subdomain of $\Omega$} for all eigenfunctions $\phi_\sigma$ {with
  $\sigma$ large enough}. To state these results we first introduce
certain conventions and notations, and we review known facts and
results from complex analysis.

In what follows, and throughout the reminder of this article, $\re^2$
is identified with the complex plane $\mathbb{C}$,
$\Omega\subset \mathbb{C}$ denotes a bounded, simply-connected open
set with analytic boundary, and $D:=\{z\in \mathbb{C}\mid |z|<1\}$
denotes the open unit disc in the complex plane. Under these
assumptions it follows from the Riemann mapping theorem~\cite{BeKr} that there is a smooth map $f:\overline{D}\to \mathbb{C}$ such that $f|_{D}:D\to \Omega$ is a biholomorphism and $|\partial_zf|>0$ on $\overline{D}$---that is to say,
$f|_D:D\to \Omega$ is a biholomorphic conformal mapping of $\Omega$ up
to and including $\partial\Omega$. {We call such a function $f$ a \emph{mapping function for $\Omega$}.} Note that,
denoting by $\partial_r$ and $\partial_\nu$ the radial derivative on
the boundary of $D$ and the normal derivative on the boundary of
$\Omega$, respectively, we have
$\partial_r=|\partial_z f|\partial_\nu$ and $|\partial_zf|>0$. Thus,
for $z\in \partial D$ the function
\begin{equation}\label{e:eigfn_circle}
u_{\sigma_j}:=\phi_{\sigma_j}\circ f
\end{equation}
satisfies,
$$
\partial_r u_{\sigma_j}(z)=|\partial_zf(z)|\partial_\nu \phi_{\sigma_j}(f(z))=|\partial_zf(z)|\sigma_j \phi_{\sigma_j}(f(z)),
$$
and, hence, the generalized Steklov eigenvalue problem
\begin{equation}
\label{e:steklovConformal}
\begin{cases}
-\Delta u_{\sigma_j}=0&\text{in }D\\
\partial_r u_{\sigma_j}=\sigma_j |\partial_z f|u_{\sigma_j}&\text{on }\partial D.
\end{cases}
\end{equation}

Finally we {introduce notation for the relevant} Fourier analysis. For $v\in C(\overline{D})$ we let
\begin{equation}\label{e:bfc}
\hat{v}(k)=\frac{1}{2\pi}\int_0^{2\pi} v(\cos \theta,\sin\theta)e^{-ik\theta}d\theta
\end{equation}
denote the ``boundary Fourier coefficients'', namely, the Fourier
coefficients of the restriction $v|_{_{\partial D}}$ of $v$ to
$\partial D$. Where notationally useful, we write
$\mathcal{F}[v] = \hat{v}$.

\begin{definition}
\label{d:tunnel} We say that the Steklov problem on $\Omega$ satisfies the tunneling condition if there is $m_0>0$ and a mapping function $f$ for $\Omega$, such that for all $K>0$ there is $C_0>0$ satisfying for any $m$
$$
|\hat{u}_{\sigma}(k)|\leq C_0^{|k-m|}\Big(\sum_{\ell=m-m_0}^{m+m_0}|\hat{u}_{\sigma}(\ell)|^2\Big)^{\frac{1}{2}},\qquad |k|\leq K\sigma.
$$ 
\end{definition}

Lemma~\ref{l:lowerBound} shows that any tunneling Steklov problem
there exist $\sigma_0>0$ so that for each $m\in\mathbf{Z}$ there is a
constant $C>0$ such that for $\sigma>\sigma_0$,
\begin{equation}
\label{e:tunnelBelow}
e^{-C\sigma}\|\hat{u}\|_{\ell^2}\leq \Big(\sum_{k= m-m_0}^{m_0}|\hat{u}(k)|^2\Big)^{\frac{1}{2}}.
\end{equation}
This estimate and its connections with similar results in quantum
mechanics motivate the ``tunneling'' terminology introduced in
Definition~\ref{d:tunnel}. To explain this, recall that $u$ is an
eigenfunction of the Dirichlet to Neumann map which is a
pseudodifferential operator on $\partial\Omega$ with symbol $|\xi|_g$
where $g$ is the metric on $\partial\Omega$~\cite[Sec. 7.11, Vol
2]{Tayl2}. Therefore, the classical problem corresponding to the
Steklov problem is the Hamiltonian flow for the Hamiltonian $|\xi|_g$
on $T^*\partial\Omega$ at energy $|\xi|_g=\sigma$---which describes
the motion of a free particle on $\partial\Omega$. The allowable
energies for this classical problem are given by $\{|\xi|_g=\sigma\}$
which, in the Fourier series representation correspond to
$\sigma = |\xi|_g\sim |k|$.  Thus, the classically forbidden region is
$\big|\sigma^{-1}|k|-1\big|>c>0$. Equation~~\eqref{e:tunnelBelow}
tells us that, in cases for which the Steklov problem on $\Omega$ is
tunneling, Steklov eigenfunctions carry positive energy even in the
classically forbidden region $\sigma^{-1}|k|\ll 1$, with an energy
value that is no smaller than exponentially decaying in
$\sigma$. (Using the estimates of~\cite{GalkToth} one can also see
that Steklov eigenfunctions carry {\em at most} exponentially small
energy in the forbidden region.)

\begin{theorem}
\label{thm:main}
Assume that the Steklov problem on $\Omega$ is tunneling and let $\sigma$ denote a Steklov eigenvalue for the set $\Omega$.
Let
\begin{equation}\label{e:delta_defs}
  \tilde{u}_{\sigma,\delta}:=u_{\sigma}|_{B(0,\delta)} = \sum_{k=-\infty}^\infty \hat{u}_{\sigma_j}(k) r^{|k|}e^{ik\theta},\qquad \tilde{u}_{\sigma_j,\delta,m}:=\sum_{|k|< m} \hat{u}_{\sigma_j}(k) r^{|k|}e^{ik\theta}.
\end{equation}
Then, there exist a constant $c>0$ such that, for each integer $N>0$,
there are constants $C_N$, $\sigma_0$, $\delta_0$, and $m_0>0$ so that
for all $0<\delta<\delta_0$, $m>m_0$, and $ \sigma_j>\sigma_0$ the
inequality
\begin{equation}
\frac{\|\tilde{u}_{\sigma,\delta} -\tilde{u}_{\sigma,\delta,m}\|_{C^N(B(0,\delta))}}{\|\tilde{u}_{\sigma,\delta}\|_{L^2(B(0,\delta))}}
\leq C_N (\delta^{m-N-m_0-1}+e^{-c\sigma})
\label{e:remainder}
\end{equation}
holds.
\end{theorem}

Letting $\{\phi_{\sigma_j}\}_{j=1}^\infty$ denote an orthonormal basis
of Steklov eigenfunctions and calling
$u_{\sigma_j}=\phi_{\sigma_j}\circ f$, Theorem~\ref{thm:main} shows in
particular that
\begin{equation}\label{e:finitely}
u_{\sigma_j}=\sum_{|k|<m}\hat{u}_{\sigma_j}(k)r^{|k|}e^{ik\sigma}+O\Big((r^{m-m_0-1}+e^{-c\sigma_j})\sqrt{\sum_{|k|<m} |\hat{u}_{\sigma_j}(k)|^2 \frac{r^{2k+1}}{2k+1}}\Big).
\end{equation}
In other words, for $r$ small, $u_{\sigma_j}$ is well approximated by
a function with finitely many Fourier modes. If there is $c>0$ such
that
$$|\hat{u}_{\sigma_j}(0)|\geq c\sqrt{\sum_{0<|k|<m}|\hat{u}_{\sigma_j}(k)|^2},$$
then we obtain 
$$
u_{\sigma_j}=\hat{u}_{\sigma_j}(0) +O((r+e^{-c\sigma_j})|\hat{u}_{\sigma_j}(0)|)
$$
and $u_{\sigma_j}$ is nearly constant on small balls centered around
0. In general, however, finitely many Fourier modes are necessary to
capture the lowest-order asymptotics, as indicated in
equation~\eqref{e:finitely}.

One of the main components of the proof of Theorem~\ref{t:nodal}, in
addition to Theorem~\ref{thm:main}, is the construction of a large
class of domains $\Omega$ for which the Steklov problem is
tunneling. To this end, we introduce some additional definitions. A
function $v\in C(D)$ will be said to be \emph{boundary-band-limited}
provided $\hat{v}(k) = 0$ except for a finite number of values of
$k\in\mathbb{Z}$. We say that a mapping function $f$ is \emph{boundary
  band limited conformal} (BBLC) if $|\partial_z f|$ is boundary
band-limited. If in addition, $|\partial_z f||_{\partial D}$ is
non-constant, we will write that $\Omega$ is BBLCN. Finally, we say
the domain $\Omega$ is BBLC (BBLCN) if and only if a {BBLC (BBLCN)
  mapping function, $f:D\to \Omega$ exists}.  We now present the main
theorem of this paper.
\begin{theorem}
\label{t:tunnel}
Assume $\Omega$ is BBLCN. Then the Steklov problem on $\Omega$ is
tunneling. 
\end{theorem}

\begin{rem}\label{disclaimer}
  It is not clear whether the elliptical and kite-shaped domains
  (equations~\eqref{ellipt_domain} and~\eqref{kite_domain}) considered
  in Figures~\ref{f:EllipseEigenvalue},~\ref{ellipses_31-81}
  and~\ref{kites_20-60} satisfy the BBLCN condition or, more
  generally, whether they have tunneling Steklov problems (we have not
  as yet been able to establish that the tunneling condition holds for
  domains that are not BBLCN). However, domain-opening observations
  such as those displayed in Figure~\ref{f:EllipseEigenvalue} and
  Section~\ref{s:numer_res}, suggest that these domains may
  nevertheless be tunneling. This and other domain-opening
  observations provide support for Conjecture~\ref{conj}
  below. (Steklov eigenfunctions on a domain which satisfies the BBLCN
  condition, and, therefore, in view of Theorem~\ref{t:tunnel}, is
  known to be tunneling, are displayed in Figure~\ref{f:approx}.)
  \end{rem}
  In view of Remark~\ref{disclaimer} we conjecture that every Steklov
  problem on an analytic domain is tunneling unless the Steklov domain
  $\Omega$ is a disc:
\begin{conjecture}
\label{conj}
Let $\Omega\subset \re^2$ be a bounded, simply-connected domain with real analytic boundary that is not equal to $B(x,r)$ for any $x\in \re^2$, $r>0$. Then the Steklov problem on $\Omega$ is tunneling.
\end{conjecture}

\subsection*{Outline of the paper}
This paper is organized as follows. Section~\ref{s:approx} shows that
arbitrary analytic, bounded, simply-connected domains can be
approximated arbitrarily closely by BBLCN domains. Then,
Sections~\ref{s:tunnel} and~\ref{s:fourier} provide proofs for
Theorems~\ref{t:tunnel} and~\ref{thm:main}, respectively. The
numerical methods used in this paper to produce accurate Steklov
eigenvalues, eigenfunctions, and associated nodal sets are presented
in Section~\ref{s:numerical}. Section~\ref{s:numer_res}, finally,
illustrates the methods with numerical results for elliptical and
kite-shaped domains.

\begin{rem}
  Throughout this article we abuse notation slightly by allowing $C$
  to denote a positive constant that may change from line to line but
  does not depend on any of the parameters in the problem. In addition
  $C_N$ is a positive constant that may change from line to line and
  depends only on the parameter $N$.
\end{rem}

\section{Approximation by tunneling domains}
\label{s:approx}

This section shows that any analytic domain can be approximated
arbitrarily closely (in a sense made precise in
Corollary~\ref{c:dense}) by a BBLCN domain. To do this, first let
$M\geq0$, $\alpha_i\in \mathbb{C}\setminus \overline{D}$ for
$i=1,\dots,N$, and let $N_i\geq 1$, $i=1,\dots M$, and let us seek
approximating BBLCN domains whose mappings
$f:\overline{D}\to \mathbb{C}$ take the form
$$
f(z)=\int_0^z p^2(w)dw,\qquad p(z)= \prod_{i=1}^M (z-\alpha_i)^{N_i}.
$$
In words: $f$ is the integral of the square of a polynomial with roots
outside $\overline{D}$.  It follows that
$$
\partial_z f=\prod_{i=1}^M(z-\alpha_i)^{2N_i},\qquad
|\partial_zf |=\prod_{i=1}^M (|z-\alpha_i|^{2})^{N}.
$$
In particular, 
$$
|\partial_z f|(e^{i\theta})=\prod_{i=1}^M(1-e^{i\theta}\overline{\alpha_i}-e^{-i\theta}\alpha_i+|\alpha_i|^2)^{N_i}
$$
which manifestly shows that ${|\partial_z f|}$ is boundary-band-limited.

We next show that an arbitrary non-vanishing analytic function on
$\overline{D}$ can be approximated by the square of a polynomial.

\begin{lemma}
\label{l:bandLimitedApprox}
Let $g:\overline{D}\to \mathbb{C}$ smooth with $g|_{D}$ analytic and $|g|>0$ on $\overline{D}$. Then, for any $\e_0>0$ and $k>0$, there are $M>0$, $\alpha_0$, $\{(\alpha_i,N_i)\}_{i=1}^M$ with $|\alpha_i|>1$, $i=1,\dots ,M$ such that 
$$
\|g- \alpha_0\prod_{i=1}^M (z-\alpha_i)^{2N_i}\|_{C^k(\overline{D})}<\e_0.
$$
\end{lemma}
\begin{proof}
Define $h:\overline{D}\to \mathbb{C}$ by
$$
h(z)=\int_0^z\frac{g'(w)}{g(w)}dw +\log(g(0))
$$
Then, since $U$ is simply-connected and $|g|>0$ on $\overline{D}$, $h$ is analytic in $D$ with smooth extension to $\overline{D}$. In addition,
$$
w(z)=e^{\frac{1}{2}h(z)}
$$
is an analytic function on $D$ such that $w^2(z)=g(z)$ and $w$ extends smoothly to $\overline{D}$. Then, for all $\e>0$, there is a polynomial $p_\e$ such that 
$$
\|w(z)-p_\e(z)\|_{C^{k}(\overline{D})}<\e\min(\|w(z)\|_{C^{k}(\overline{D})},1)
$$
In particular, since $|g|>c>0$ on $\overline{D}$, for $0<\e$ small enough, $p_\e$ has no zeros in $\overline{D}$. Hence, 
$$
p_\e=\beta_0\prod_{i=1}^M(z-\beta_i)^{N_i}
$$
for some $|\beta_0|>0$, $|\beta_i|>1$, $i=1,\dots, M$. Multiplying by $w+p_\e$, we have
\begin{align*}
\|g(z)-p^2_\e(z)\|_{C^k(\overline{D})}&=\|(w-p_\e)(w+p_\e)\|_{C^k(\overline{D})}\\
&\leq C_k\|(w-p_\e)\|_{C^{k}(\overline{D})}\|(w+p_\e)\|_{C^{k}(\overline{D})}\\
&\leq C_k\e(2+\e)\|w\|_{C^{k}(\overline{D})}
\end{align*}
Choosing $\e=\frac{\e_0}{C_k}\min(\frac{1}{3\|w\|_{C^{k}(\overline{D})}},1)$ proves the result with $\alpha_0=\beta_0^2$ and $\alpha_i=\beta_i$.
\end{proof}

This result can be used to approximate any analytic domain by a BBLCN
domain:
\begin{cor}
\label{c:dense}
For any analytic, bounded, simply-connected domain $\Omega$, $k>0$, and $\e_0>0$ there is a BBLCN domain $\Omega_{\e_0}$ and $g_{\e_0}\in C^\infty(\partial\Omega)$ such that with $\nu$ the outward unit normal to $\Omega$,
\begin{equation}\label{e:pert}
 \partial\Omega_{\e_0}=\{ x+\nu g_{\e_0}(x)\mid x\in \partial\Omega\},\qquad \|g_{\e_0}\|_{C^k(\partial\Omega)}<\e_0.
\end{equation}
\end{cor}
\begin{proof}
Since $\Omega$ is analytic, there is $f:\overline{D}\to \mathbb{C}$ analytic such that $f|_{D}:D\to \Omega$ is a biholomorphism and $|\partial_z f|>0$ on $D$. Moreover, by~\cite{BeKr}, $\partial_zf$ has a smooth extension to $\overline{D}$. Then, applying Lemma~\ref{l:bandLimitedApprox} with $g=\partial_zf(z)$ gives 
$$p_\e=\alpha_0\prod_{i=1}^M(z-\alpha_i)^{N_i}$$ a polynomial with no roots in $\overline{D}$ such that
$$
\|\partial_zf(z)-p^2_\e(z)\|_{C^{\max(k,1)}(\overline{D})}<\e.
$$
Note also that adjusting $p$ if necessary we may assume that the
restriction of $|p_\e|$ to $\partial D$ is not constant.  Then,
defining
\begin{equation}\label{e:feps}
  f_\e:=\int_0^z p^2_\e(w)dw+f(0)
\end{equation}
we have
$$
\|f_\e-f\|_{C^{\max(k+1,2)}(\overline{D})}<\e,\qquad \partial_zf_\e=p^2_\e,
$$
so that $\big|\partial_z f_\e\big||_{_{\partial D}}$ is non-constant
and band limited.  Moreover, since $f$ is a biholomorphism, for $\e>0$
small enough, $f_\e$ is also a biholomorphism. We next show that since
$\|f_\e-f\|_{C^{\max(k+1,2)}(\overline{D})}<\e$, for $\e>0$ small
enough the curve
$$
\partial\Omega_\e=\{f_\e(z)\mid |z|=1\}
$$
can be expressed in the form~\eqref{e:pert}. To do this let 
$$
F(t,\theta,\omega,s)=f(e^{i\theta})-tf_\e(e^{i(\omega+\theta)})-(1-t)f(e^{i(\omega+\theta)})-sf'(e^{i\theta})e^{i\theta}
$$
and note that $F(1,\theta,\omega,s)=0$ if and only if
$$
f_\e(e^{i(\omega+\theta)})=f(e^{i\theta})\pm s\nu(\theta).
$$
Therefore, we aim to find $s=s(\theta)$ and $\omega=\omega(\theta)$ such that $F(1,\theta,\omega(\theta),s(\theta))=0$.
Note that
$$
\begin{aligned}
\partial_sF&=-f'(e^{i\theta})e^{i\theta}\\
\partial_{\omega}F&= -ie^{i(\omega+\theta)}(f'(e^{i(\omega+\theta)})+t(f_\e'(e^{i(\omega+\theta)})-f'(e^{i(\omega+\theta)}))
\end{aligned}
$$
In particular, 
$$
\partial_{\omega}F=i\partial_sF+O(\e)+O(|\omega|).
$$
Therefore, there is $\delta>0$, $\e_0>0$ such that for $0<\e<\e_0$,
$|\omega_0|<\delta$, $t_0\in(-1,2)$, and $s_0\in[-1,1]$ if
$F(t_0,\theta_0,\omega_0,s_0)=0$, then for $|\omega_0|<\delta$ and
$|t-t_0|<\delta$, $\omega=\omega(t,\theta)$ and $s=s(t,s)$ are the
unique solutions of $F(t,\theta,\omega,s)=0$. In particular, since
$F(0,\theta,0,0)=0$, the solutions $s=s(t,\theta)$ and
$\omega=\omega(t,\theta)$ can be continued as functions of $t$ as long
as $|\omega(t,\theta)|$ remains small. 

We next note that
$$
\begin{pmatrix} \partial_t \omega\\\partial_ts\end{pmatrix}=\begin{pmatrix}\partial_\omega F&\partial_sF\end{pmatrix}^{-1}\partial_tF=O(\|f_\e-f\|_{L^\infty})=O(\e),
$$
and, therefore,
$$
|\omega(t,\theta)|+|s(t,\theta)|\leq \int_0^{t}|\partial_t\omega(r,\theta)|+|\partial_ts(r,\theta)|dr\leq C t\e.
$$
Hence for $\e$ small enough the solutions $\omega(t,\theta)$ and
$s(t,\theta)$ continue to $t=1$ and satisfy
$$
|\omega(1,\theta)|+|s(1,\theta)|\leq C\e.
$$
Again, using the implicit function theorem, this implies that
$\omega(\theta):=\omega(1,\theta)$ and $s(\theta):=s(1,\theta)$ are
$2\pi$-periodic. Differentiating $k$ times now yields
$$
|\partial_\theta^ks|\leq C_k\e,
$$
finishing the proof by setting $g_{\e_0}=\pm s$ and shrinking $\e>0$ as necessary. (Here the $\pm$ corresponds to whether $f(e^{i\theta})$ is positively ($-$) or negatively ($+$) oriented.)
\end{proof}

\begin{rem}\label{r:fe}
  Since the map $f_\e$ in equation~\eqref{e:feps} may send 0 to a point
  $z_0$ close to the boundary, it is interesting to see how the
  Steklov eigenfunctions rearrange their nodal sets in such a way that
  Theorems~\ref{t:nodal} and~\ref{thm:main} are satisfied on the image
  of $f_\e$. To demonstrate this let $|a|<1$, consider the
  biholomorphic function $f(z):=\frac{z-a}{\bar{a}z-1}$, and let
  $f_\e$ denote the approximant of $f$ given by equation~(\ref{e:feps})
  with
  \begin{equation}\label{e:fe}
    p_\e(z)=i \sqrt{1- |a|^2}\sum_{j=0}^N(\bar{a}w)^j\quad
    \mbox{with $N=20$ and $a=0.8$.}
\end{equation}  
(This polynomial was obtained as the $N$-th order Taylor polynomial of
$\sqrt{\partial_z f}$.) In this case, according to Theorems~\ref{t:nodal}
and~\ref{thm:main}, the Steklov eigenfunctions should be slowly
oscillating in a $\sigma$ independent neighborhood of $z_0$.
Figure~\ref{f:approx} displays corresponding Steklov eignfunction or
various orders as well as a typical eigenfunction for the exact
disc. Note the dramatic change that arises in the Steklov
eigenfunctions from a barely visible boundary perturbation of the
disc.
\end{rem}

\begin{figure}
\includegraphics[width=0.45\textwidth]{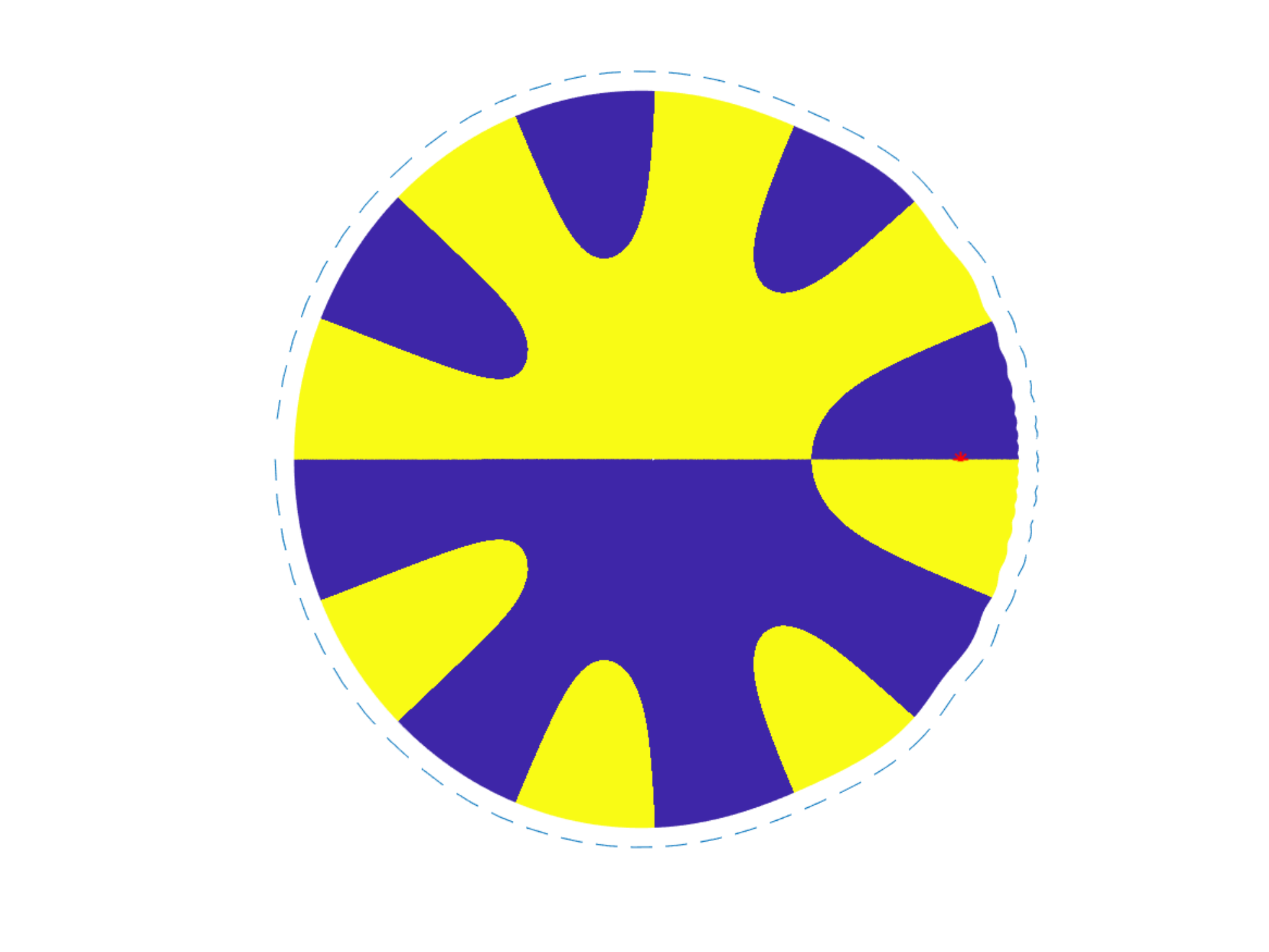}
\includegraphics[width=0.45\textwidth]{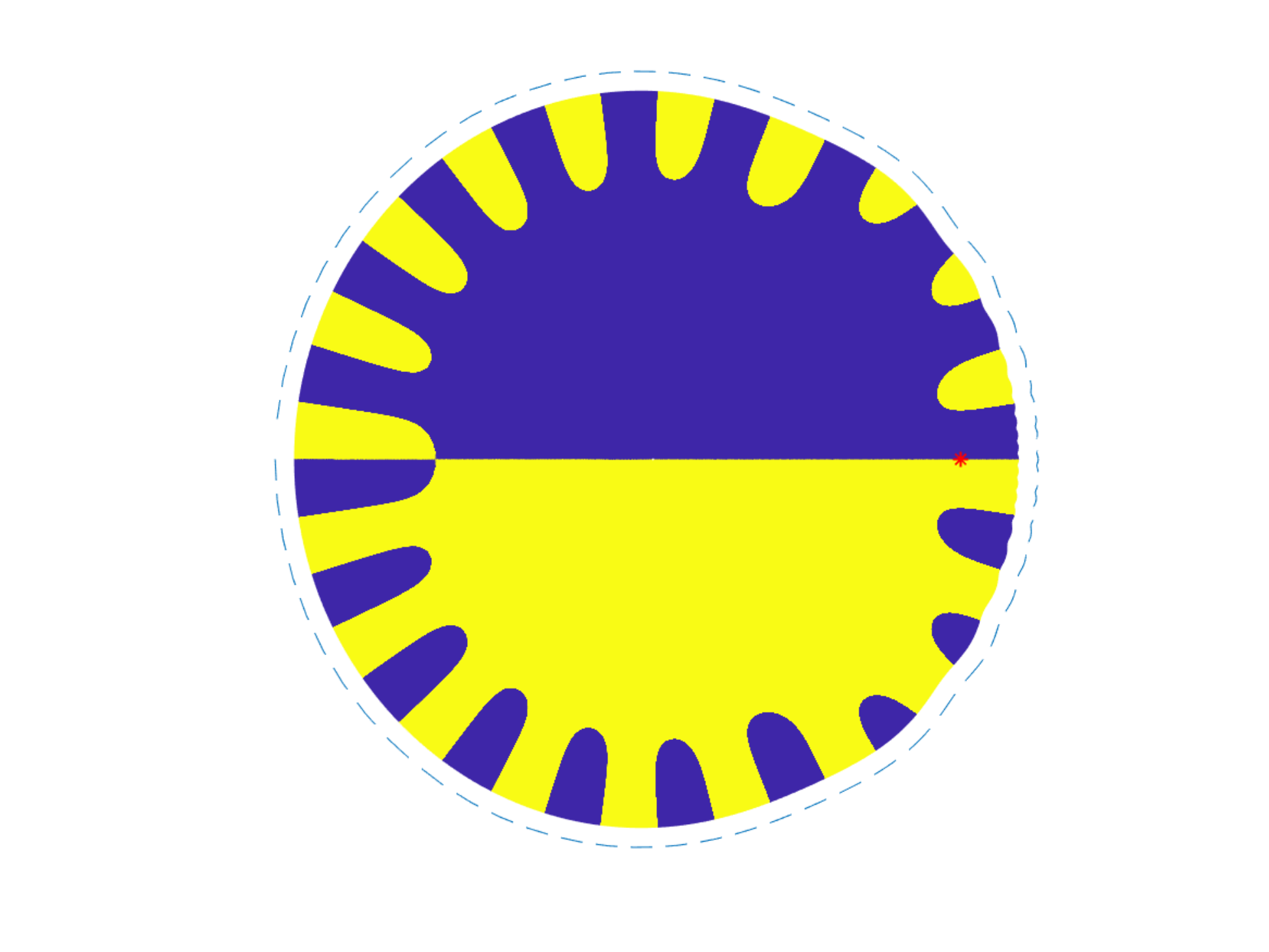}
\includegraphics[width=0.45\textwidth]{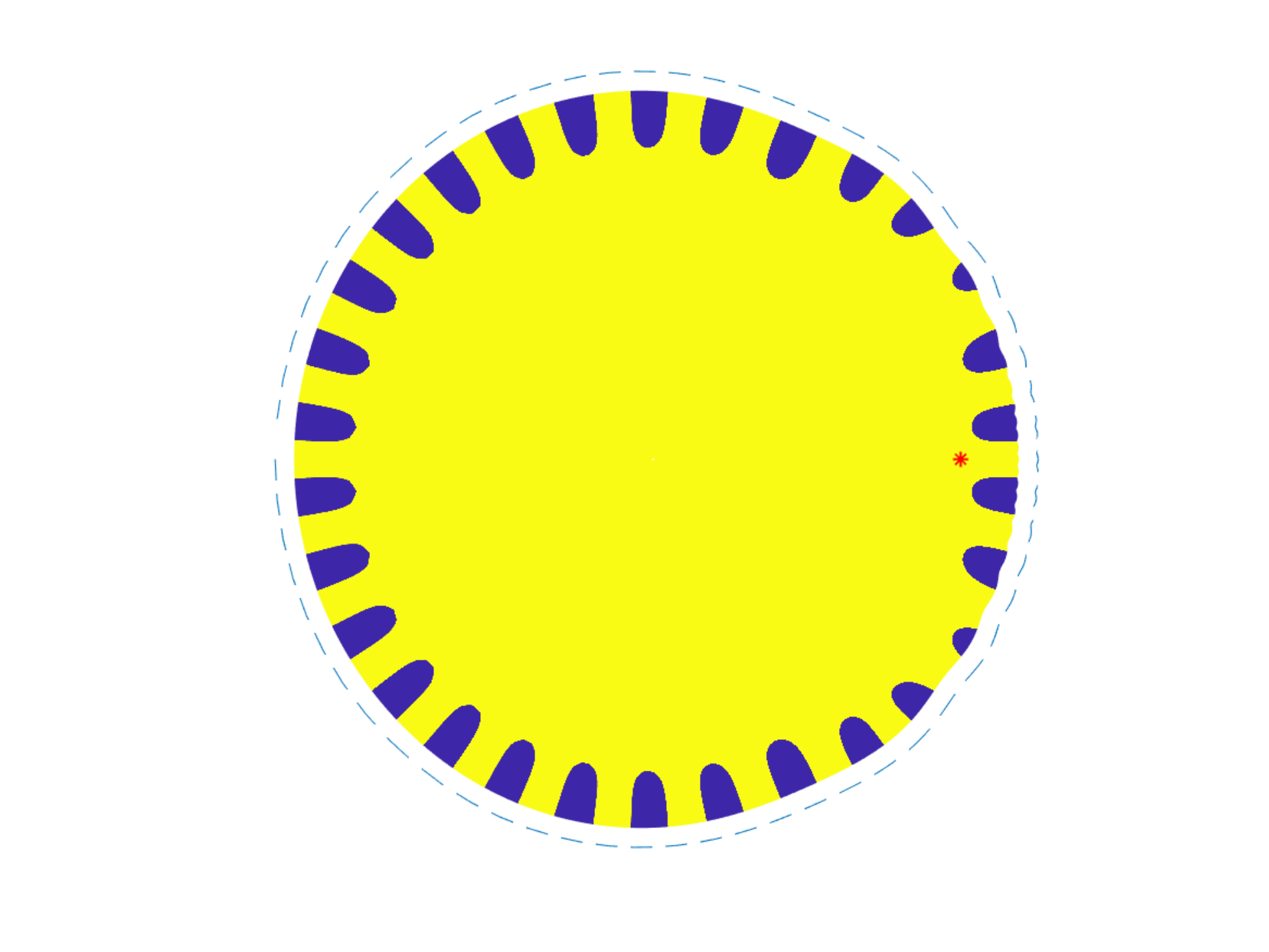}
\includegraphics[width=0.45\textwidth]{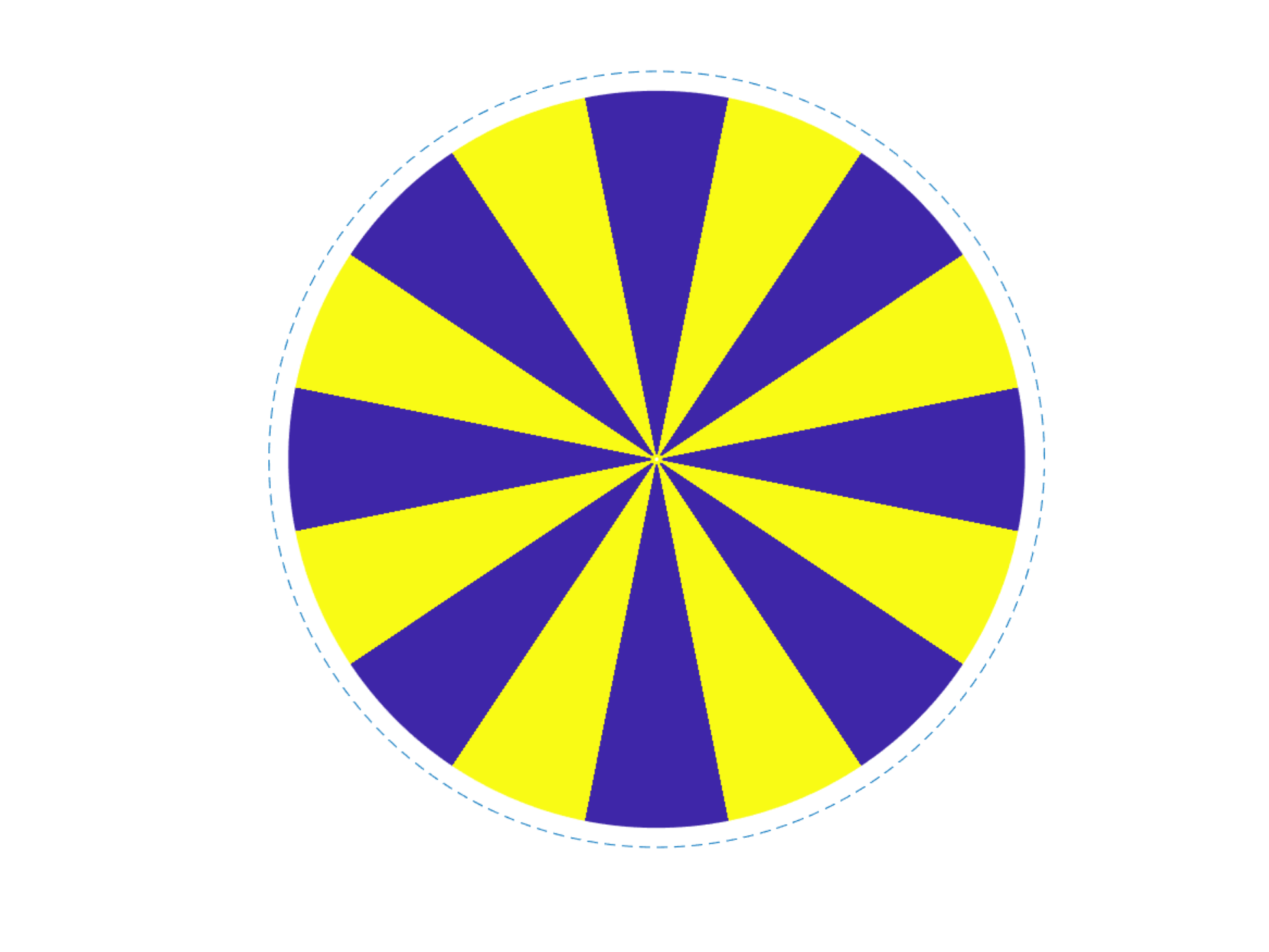}
\caption{\label{f:approx} Steklov eigenfunctions on the domain
  $\Omega$ whose mapping function, which is given by
  equation~\eqref{e:fe}, maps the center of the disk to the point
  $z_0 = (0.8,0)$ (marked by red asterisks in the figures). The
  corresponding Steklov eigenvalues are given by $\sigma_{16}=7.9642$
  (top left), $\sigma_{40}=19.8173$ (top right), and
  $\sigma_{60}=29.8197$ (bottom left). Note that, according to
  Corollary~\ref{c:dense} the set $\Omega$ is a BBLCN approximation to
  the disk. As predicted by Theorem~\ref{thm:main}, oscillations avoid
  a region around $z_0$ for high $\sigma$. The bottom-right image
  displays a typical eigenfunction on the exact disc. Note the
  dramatic change that arises in the Steklov eigenfunctions from a
  barely visible boundary perturbation of the disc.}
\end{figure}

\section{BBLCN domains and tunneling Steklov problems}
\label{s:tunnel}
This section presents a proof of Theorem~\ref{t:tunnel}. In
preparation for that proof, let $\Omega\subset \mathbb{C}$ be a BBLCN
domain, and denote by $f$ the corresponding mapping function. Define
$$
\mathcal{F}\left[ \big|\partial_zf\big|\right](n):=a_n,\quad n\in\mathbb{Z} \quad\left( a_0:=\frac{1}{2\pi}\int_0^{2\pi}|\partial_z f(e^{i\theta})|\,d\theta>0\right).
$$
Since $\Omega$ is a BBLCN domain, the function
$\big|\partial_z f\big||_{_{\partial D}}$ is band limited and
$\big|\partial_z f\big||_{_{\partial D}}$ is not identically
constant. It follows that
$$
m_0:=\sup \{|n|: |a_n|\neq 0\}
$$
satisfies $1\leq m_0<\infty$.

Denoting by $\hat{u}(n)$ the boundary Fourier coefficients of an
eigenfunction $u$, the corresponding boundary Fourier coefficients of
$\partial_r u$ are given by $|n|\hat{u}(n)$. Thus, a solution
to~\eqref{e:steklovConformal} is uniquely determined as an $\ell^2$
solution to the equation
\begin{equation}
 \label{e:steklovFourier1}
 |n|\hat{u}(n)=\sigma \mathcal{F} \left[u\big |\partial_z f\big |\right ](n)\quad n\in\mathbb{Z}.
\end{equation}
In what follows we may, and do, assume that solutions $\hat{u}$ have
$\ell^2$-norm equal to one.

\subsection*{Proof of Theorem~\ref{t:tunnel}}

Since
$$
\mathcal{F}\left[ u\big|\partial_z f\big|\right]=\sum_{m} a_{m}\hat{u}(n-m),
$$
it follows that~\eqref{e:steklovFourier1} can be re-expressed in the form
\begin{equation}
\label{e:steklovFourier}
|n|\hat{u}(n)=\sum_m \sigma a_m\hat{u}(n-m).
\end{equation}
From~\eqref{e:steklovFourier} we obtain
$$
a_{-m_0}\hat{u}(n+m_0)=\frac{|n|}{\sigma}\hat{u}(n)-\sum_{m\neq -m_0}a_m\hat{u}(n-m),
$$
and, then, for all $|n|\leq K\sigma $,
\begin{align*}
|\hat{u}(n+m_0)|&\leq |a_{-m_0}|^{-1}\Big(\frac{||n|-\sigma a_0|}{\sigma}|\hat{u}(n)|+\sum_{m\neq 0,-m_0}|a_m||\hat{u}(n-m)|\Big)\\
&\leq |a_{-m_0}|^{-1}\Big(\frac{||n|-\sigma a_0|}{\sigma}|\hat{u}(n)|+\sum_{\substack{m=-m_0+1\\m\neq 0}}^{m_0}|a_m||\hat{u}(n-m)|\Big)\\
&\leq |a_{-m_0}|^{-1}\max(K,\|a_m\|_{\ell^\infty})\sum_{k=n-m_0}^{n+m_0-1}|\hat{u}(k)|.
\end{align*}
The second inequality follows from the fact that $a_n\equiv 0$ for
$|n|\geq m_0$, while the third one results from the relation $a_0>0$
and the positivity, $\sigma > 0$, of all nontrivial eigenvalues
$\sigma$, which imply that
$$
||n|-\sigma a_0|\leq\max (|n|,\sigma|a_0|){\leq \sigma(\max(K,\|a_m\|_{\ell^\infty}))}.
$$

Making an identical argument, but solving for $\hat{u}(n-m_0)$, and using that $|a_{m_0}|=|a_{-m_0}|\neq 0$, we have for all $|n|\leq K\sigma$,
\begin{equation}
\label{e:est1}
\begin{aligned}
|\hat{u}(n+m_0)|&\leq |a_{m_0}|^{-1}\max(2,\|a_m\|_{\ell^\infty})\sum_{k=n-m_0}^{n+m_0-1}|\hat{u}(k)|,\\
|\hat{u}(n-m_0)|&\leq |a_{m_0}|^{-1}\max(2,\|a_m\|_{\ell^\infty})\sum_{k=n-m_0+1}^{n+m_0}\!\!\!\!\!\!|\hat{u}(k)|.
\end{aligned}
\end{equation}

We now use equation~\eqref{e:est1} to prove the first half of our
tunneling estimate.
\begin{lemma}
\label{l:induct}
Let $m\in \mathbb{Z}$, $K>0$, and
$$
A_m:=\Big(\sum_{k=m-m_0}^{m+m_0}|\hat{u}(k)|^2\Big)^{\frac{1}{2}}.
$$
Then, there exists $C_0>0$ so that for all $\sigma>0$ and for
$-K\sigma \leq n+m\leq K\sigma$ we have
\begin{equation}
\label{e:induct}
|\hat{u}(n+m)|\leq C_0^{|n|}A_m.
\end{equation}
\end{lemma}
\begin{proof}
We will assume $m\geq 0$ since the other case follows similarly. The cases of $n=-m_0,\dots,m_0$ are clear if we take $C_0\geq 1$. Suppose~\eqref{e:induct} holds for $-m_0\leq n\leq \ell$ with $m_0\leq \ell$. Then, by~\eqref{e:est1}, 
\begin{align*}
|\hat{u}(m+\ell+1)|&\leq |a_{m_0}|^{-1}\max(K,\|a_m\|_{\ell^\infty})\sum_{k=\ell-2m_0+1}^{\ell}|\hat{u}(k+m)|\\
&\leq |a_{m_0}|^{-1}\max(K,\|a_m\|_{\ell^\infty})\sum_{k=\ell-2m_0+1}^{\ell}C_0^{|k|}A
\end{align*}
{Now, if $m_0\leq \ell<2m_0$, then
\begin{align*}
|\hat{u}(m+\ell+1)|&\leq |a_{m_0}|^{-1}\max(K,\|a_m\|_{\ell^\infty})\Big(\sum_{k=0}^{\ell}C_0^{k}+\sum_{k=1}^{2m_0-\ell-1}C_0^k\Big)A\\
&\leq |a_{m_0}|^{-1}\max(K,\|a_m\|_{\ell^\infty})\Big(\frac{C_0^{\ell+1}-1+C_0^{2m_0-\ell+1}-C_0}{C_0-1}\Big)A
\end{align*}
In particular, taking
$$
C_0\geq 2|a_{m_0}|^{-1}\max(K,\|a_m\|_{\ell^\infty})+1
$$
we have 
$$
|\hat{u}(m+\ell+1)|\leq C_0^{\ell+1}A.
$$
Next, if $2m_0\leq \ell$, then} 
\begin{align*}
|\hat{u}(\ell+m+1)|&\leq |a_{m_0}|^{-1}\max(K,\|a_m\|_{\ell^\infty})A\frac{ C_0^{\ell+1}-C_0^{\ell-2m_0+1}}{C_0-1}
\end{align*}
Taking ${C_0\geq 2 |a_{m_0}|^{-1}\max(K,\|a_m\|_{\ell^\infty})+1}$
completes the proof for $-m_0\leq n\leq K\sigma-m$. 

An almost identical argument gives the $-K\sigma-m \leq n\leq 0$ case.
\end{proof}

\section{Analysis of Tunneling Steklov Problems}
\label{s:fourier}

The proof of Theorem~\ref{thm:main} now follows in two steps. First,
we show that, for eigenfunctions of any tunneling Steklov problem, the
boundary Fourier coefficients of low frequency contain a mass no
smaller than exponential in $\sigma$. To finish the proof, we use the
fact that the harmonic extension of $e^{in\theta}$ decays exactly as
$r^{|n|}$. Examining the solution on the ball of radius $\delta>0$ for
some $\delta$ small enough, it will be shown that the low frequencies
dominate the behavior of $u$.

\begin{lemma}
\label{l:lowerBound}
Suppose that $\Omega$ has tunneling Steklov problem. Then there exist $\sigma_0>0$ so that for all $m>0$ there is $C>0$ such that for $\sigma>\sigma_0$,
$$
e^{-C\sigma}\|\hat{u}\|_{\ell^2}\leq \Big(\sum_{k= m-m_0}^{m_0}|\hat{u}(k)|^2\Big)^{\frac{1}{2}}=:A_m.
$$
\end{lemma}
\begin{proof}
First, note that by e.g.~\cite[Corollary 1.3]{GalkToth}, for $\sigma>3m$ there is $C>0$ so that
$$
\sum_{|k-m|\leq 2\sigma}|\hat{u}(k)|^2\geq\|\hat{u}\|_{\ell^2}^2(1-Ce^{-\sigma/C})).
$$
By Lemma~\ref{l:induct}
$$
\sum_{|k-m|\leq 2\sigma}|\hat{u}(k)|^2\leq \sum_{0\leq k-m\leq 2\sigma} C_0^{2k} A_m^2+\sum_{-2\sigma \leq k-m<0}C_0^{2|k|}A_m^2\leq 2 \frac{2C_0^{4\sigma+2}-1}{C^2_0-1}A_m^2
$$
In particular, 
$$
\frac{C^2_0-1}{2(2C_0^{4\sigma+2}-1)}\|\hat{u}\|_{\ell^2}^2(1-Ce^{-C\sigma})\leq A_m^2=\sum_{k=-m_0}^{m_0}|\hat{u}(k)|^2.
$$
Taking $\sigma_0$ large enough so that $Ce^{-C\sigma}\leq \frac{1}{2}$, finishes the proof.
\end{proof}

\begin{proof}[Proof of Theorem~\ref{thm:main}]
  In what follows we utilize the definitions~\eqref{e:delta_defs} for
  a given eigenvalue $\sigma_j = \sigma$, and, for that eigenvalue we
  denote $\hat{u}(k) = \hat{u}_{\sigma_j}(k) = \hat{u}_{\sigma}(k)$.
  Then, applying the relation
\begin{equation}
\label{e:L2norm}
\int_{B(0,\delta)} |\sum_k b_k r^{|k|}e^{ik\theta}|^2=  \sum_k |b_k|^2\frac{2\pi \delta^{2|k|+{2}}}{2|k|+{2}},
\end{equation}
which is valid for all sequences
$\{b_k\}_{k\in \mathbb{Z}}\subset \mathbb{C}$, to the right-hand
equation in~\eqref{e:delta_defs}, for $m\geq m_0$ we obtain
\begin{equation}
\label{e:L2bound}\|\tilde{u}_{\sigma,\delta,m}\|^2_{L^2}= \sum_{|k|\leq m}\frac{2\pi \delta^{2|k|+{2}}}{2k+{2}}|\hat{u}(k)|^2\geq 2\pi \frac{\delta^{2m_0+{2}}}{2m_0+{2}}\sum_{|k|\leq m_0}|\hat{u}(k)|^2=2\pi \frac{\delta^{2m_0+{2}}}{2m_0+{2}}A^2.
\end{equation}

To estimate the error in approximating $u_{\sigma,\delta}$ by $\tilde{u}_{\sigma,\delta,m}$, first note that 
\begin{align*}
\Big\|\sum_{|k|\geq 2\sigma} \hat{u}(k)r^{|k|}e^{ik\theta}\Big\|_{C^N(B(0,\delta))}&\leq \sum_{|k|\geq 2\sigma} |\hat{u}(k)|\cdot \|r^{|k|}e^{ik\theta}\|_{C^N(B(0,\delta))}\\
&\leq \Big(\sum_{|k|\geq 2\sigma} |\hat{u}(k)|^2\Big)^{\frac{1}{2}}\Big(\sum_{|k|\geq 2\sigma} \|r^{|k|}e^{ik\theta}\|^2_{C^N(B(0,\delta))}\Big)^{\frac{1}{2}}\\
&\leq \Big(\sum_{|k|\geq 2\sigma} |\hat{u}(k)|^2\Big)^{\frac{1}{2}}\Big(\sum_{|k|\geq 2\sigma} k^{2N}\delta^{2k-2N}\Big)^{\frac{1}{2}}\\
&\leq C_N\|\hat{u}\|_{\ell^2}\delta^{-N}{\sigma^N}\delta^{2\sigma}.
\end{align*}
Applying Lemma~\ref{l:lowerBound} with $m=0$, {and absorbing the $\sigma^N$ into the exponential factor} we then obtain
$$
\Big\|\sum_{|k|\geq 2\sigma} \hat{u}(k)r^{|k|}e^{ik\theta}\Big\|_{C^N(B(0,\delta))}\leq C_N\delta^{-N}\delta^{2\sigma}e^{C\sigma}A$$
where
$$
A:=\Big(\sum_{k=-m_0}^{m_0}|\hat{u}(k)|^2\Big)^{\frac{1}{2}},
$$

We can now estimate
\begin{align*}
\|\sum_{|k|\geq m} \hat{u}(k) r^{|k|}e^{ik\theta}\|_{C^N(B(0,\delta))}&\leq \sum_{m\leq |k| <  2\sigma} \|\hat{u}(k) r^{|k|}e^{ik\theta}\|_{C^N(B(0,\delta))} +\|\sum_{|k|\geq 2\sigma} \hat{u}(k) r^{|k|}e^{ik\theta}\|_{C^N(B(0,\delta))}\\
&\leq \sum_{m\leq |k|< 2\sigma} |\hat{u}(k)|\cdot \| r^{|k|}e^{ik\theta}\|_{C^N(B(0,\delta))} +C_N\delta^{-N}\delta^{2\sigma}e^{C\sigma}A\\
\end{align*}
Thus, using the definition of tunneling (Definition~\ref{d:tunnel}), we obtain
\begin{align*}
\|\sum_{|k|\geq m} \hat{u}(k) r^{|k|}e^{ik\theta}\|_{C^N(B(0,\delta))}
&\leq C_N\delta^{m-N}A\sum_{m\leq |k|< 2\sigma} C_0^{|k|}|k|^N\delta ^{|k|-m}+C_N\delta^{-N}\delta^{2\sigma}e^{C\sigma}A\\
&\leq C_N\delta^{m-N} A+C_N\delta^{-N}\delta^{2\sigma}e^{C\sigma}A
\end{align*}
provided that $\delta<\frac{1}{2}C_0^{-1}$.  
Therefore, using~\eqref{e:L2bound},
$$
\frac{\|\tilde{u}_{\sigma,\delta} -\tilde{u}_{\sigma,\delta,m}\|_{C^N(B(0,\delta))}}{\|\tilde{u}_{\sigma,\delta}\|_{L^2(B(0,\delta))}} \leq C_N\delta^{m-N-m_0-1}+C_N\delta^{2\sigma-N-m_0-1}e^{C\sigma}.
$$
Thus,
choosing $\delta>0$ such that $\delta < e^{-2C}$ and taking
$\sigma_0>N{+m_0+1}$ the claim follows.
\end{proof}

We can now present a proof of Theorem~\ref{t:nodal}.

\begin{proof}[Proof of Theorem~\ref{t:nodal}]
  From Corollary~\ref{c:dense} we know that there exists a tunneling
  domain $\Omega_1\subset\mathbb{C}$ satisfying~\eqref{e:pert_dom} for
  the given value $\varepsilon >0$. Let $\sigma_0$ be as in
  Theorem~\ref{thm:main}. Clearly, it suffices to prove the statement
  of the theorem for $\sigma> \sigma_0$, since for
  $\sigma\leq \sigma_0$ the statement follows from the fact that there
  are finitely many Steklov eigenvalues below $\sigma_0$ and that
  $\psi_{\sigma}$ cannot vanish in any open set. Therefore, we may and
  do assume $\sigma > \sigma_0$ along with the other assumptions in
  Theorem~\ref{thm:main}, so that, in particular,
  inequality~\eqref{e:remainder} holds. In what follows we write
\begin{equation}\label{Lpdelta-notation}
L^2(B(0,\delta)) = L^2_\delta\quad\mbox{and}\quad L^\infty(B(0,\delta)) = L^\infty_\delta
\end{equation}

  Fixing $m\geq m_0+2$, and letting $\tilde{u}_{\sigma,\delta}$ and
  $\tilde{u}_{\sigma,\delta,m}$ be given by~\eqref{e:delta_defs}
  (with $u_{\sigma}$ related to $\phi_\sigma$ via
  ~\eqref{e:eigfn_circle}) we note that
$$
\|\tilde{u}_{\sigma,\delta,m}\|_{L^\infty_\delta}\geq \frac{\|\tilde{u}_{\sigma,\delta,m}\|_{L^2_\delta}}{\sqrt{\pi} \delta}.
$$

It follows that there exists $x_0\in B(0,\delta)$ such that 
\begin{equation}
\label{e:bigness}
|\tilde{u}_{\sigma,\delta,m}(x_0)|\geq  \frac{\|\tilde{u}_{\sigma,\delta,m}\|_{L^2_\delta}}{\sqrt{\pi} \delta}.
\end{equation}
Now, since $\|\tilde{u}_{\sigma,\delta,m}\|_{C^1}\leq C_{m,\delta}\|\tilde{u}_{\sigma,\delta,m}\|_{L^2}$, it follows from~\eqref{e:bigness} that there is $r_{m,\delta}\in\mathbb{R}$, $0< r_{m,\delta}<\delta$  (in particular, independent of $\sigma$)
such that
\begin{equation}\label{e:strict_mv}|\tilde{u}_{\sigma,\delta,m}(x)|>\frac{\|\tilde{u}_{\sigma,\delta,m}\|_{L^2_\delta}}{2\sqrt{\pi} \delta},\qquad x\in B(x_0,r_{m,\delta}).
\end{equation}
But, since $m\geq m_0+2$, the estimate~\eqref{e:remainder} with $N=0$
yields
\begin{equation}\label{e:m-estimate}
|\tilde{u}_{\sigma,\delta,m}(x)|\leq |\tilde{u}_{\sigma,\delta}(x)| + |\tilde{u}_{\sigma,\delta}(x)-\tilde{u}_{\sigma,\delta,m}(x)| \leq C_0(\delta+e^{-c\sigma}) \|\tilde{u}_{\sigma,\delta} \|_{L^2_\delta} + |\tilde{u}_{\sigma,\delta}(x)|
\end{equation}
and 
\begin{equation}\label{e:m-L2-estimate}
 \|\tilde{u}_{\sigma,\delta} \|_{L^2_\delta}
\leq
 \|\tilde{u}_{\sigma,\delta,m}\|_{L^2_\delta} + \|\tilde{u}_{\sigma,\delta} - \tilde{u}_{\sigma,\delta,m}\|_{L^2_\delta} \leq \|\tilde{u}_{\sigma,\delta,m}\|_{L^2_\delta} + \sqrt{\pi}\delta C_0(\delta+e^{-c\sigma}) \|\tilde{u}_{\sigma,\delta} \|_{L^2_\delta}.
\end{equation}
(To establish the rightmost inequality in~\eqref{e:m-L2-estimate} the
relation
$\|\tilde{u}_{\sigma,\delta} -
\tilde{u}_{\sigma,\delta,m}\|_{L^2_\delta}\leq
\sqrt{\pi}\delta\|\tilde{u}_{\sigma,\delta} -
\tilde{u}_{\sigma,\delta,m}\|_{L^\infty_\delta}$ was used before the
inequality~\eqref{e:remainder} was applied.)
From~\eqref{e:m-L2-estimate} we obtain
\begin{equation}\label{e:m-lower-L2}
  \|\tilde{u}_{\sigma,\delta,m}\|_{L^2_\delta}\geq \|\tilde{u}_{\sigma,\delta} \|_{L^2_\delta} -
  \sqrt{\pi}\delta C_0(\delta+e^{-c\sigma}) \|\tilde{u}_{\sigma,\delta} \|_{L^2_\delta}.
\end{equation}

It follows from~\eqref{e:strict_mv}, \eqref{e:m-estimate}
and~\eqref{e:m-lower-L2} that
\begin{equation}\label{e:m-inf-estimate}
  C_0(\delta+e^{-c\sigma}) \|\tilde{u}_{\sigma,\delta} \|_{L^2_\delta} + |\tilde{u}_{\sigma,\delta}(x)|>\frac{\|\tilde{u}_{\sigma,\delta,m}\|_{L^2_\delta}}{2\sqrt{\pi} \delta} \geq \frac{\|\tilde{u}_{\sigma,\delta}\|_{L^2_\delta}}{2\sqrt{\pi} \delta}-
  \frac{C_0}{2}(\delta+e^{-c\sigma}) \|\tilde{u}_{\sigma,\delta} \|_{L^2_\delta},\quad x\in B(x_0,r_{m,\delta}),
\end{equation}
and, therefore
\begin{equation}\label{e:m-inf-estimate_2}
  |\tilde{u}_{\sigma,\delta}(x)|> \|\tilde{u}_{\sigma,\delta}\|_{L^2_\delta}\left(\frac{1}{2\sqrt{\pi} \delta}-
  \frac{3C_0}{2}(\delta+e^{-c\sigma}) \right),
\qquad x\in B(x_0,r_{m,\delta}).
\end{equation}
Taking $\delta_1$ sufficiently small and $\delta\leq\delta_1$ the
inequality
$$
\frac{3C_0}{2}(\delta+e^{-c\sigma_0})<\frac{1}{2\sqrt{\pi}\delta}
$$
holds, and it therefore follows that for a certain constant $D>0$ we have
$$
|u_{\sigma,\delta}(x)|>\frac{D\|\tilde{u}_{\sigma,\delta}\|_{L^2_\delta}}{\delta}\quad\mbox{for}\quad x\in B(x_0,r_{m,\delta})
$$
provided $\delta<\delta_1$. In particular,
$$
|\phi_{\sigma}(x)|>0,\qquad x\in f(B(x_0,r_{m,\delta})).
$$
Since the derivative of $f$ never vanishes, for $\delta <\delta_1$ and
for a certain $E>0$ there is a ball $\mc{B}$ of radius $Er_{m,\delta}$
such that $\phi_{\sigma}$ does not vanish on $\mc{B}$. The proof is
now complete.

\end{proof}

	\begin{figure}[!h]		
  \centering		
  \includegraphics[width=.45\textwidth]{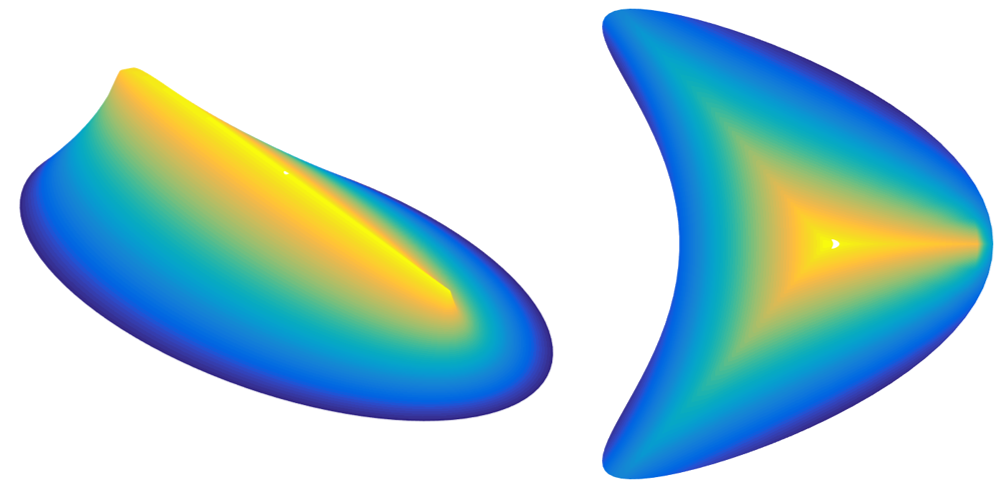}		
  \caption{The function $\lambda$ for an ellipse (left) and a kite-shaped domain (right).\label{sigmas}}		
\end{figure}

\section{Numerical Formulation}
\label{s:numerical}

\subsection{Integral representation}
Let $\Omega\subset\mathbb{R}^2$ denote a domain with, say, a
$C^2$ boundary, and let
\[
  S[\phi](x) := \int \limits_{\partial \Omega} G(x,y) \phi(y) dS(y),\quad x\in\mathbb{R}^2 ,\quad G(x,y) = -\frac{1}{2\pi}\log|x-y|,
\]
denote the Single Layer Potential (SLP) for a given density
$\phi:\partial \Omega\to \mathbb{R}$ in a certain Banach space $H$ of
functions. Both Sobolev and continuous spaces $H$ of functions lead to
well developed Fredholm theories in this
context~\cite{Kress2010,mclean2000strongly}. It is useful to recall
that, as shown e.g. in the aforementioned references, the limiting
values of the potential $S$ and its normal derivative on
$\partial\Omega$ can be expressed in terms of well known ``jump
conditions'' that involve the single and double layer boundary
integral operators
\[
  \mathcal{S}[\phi](x):= \int_{\partial \Omega} G(x,y)\phi(y) ds(y) \quad
  \quad\mbox{and}\quad \mathcal{T}[\phi](x):= \int_{\partial \Omega}
  \frac{\partial G(x,y)}{\partial {\nu}(x)}\phi(y)ds(y),\quad
  x\in\partial\Omega,
\]
respectively. 

In view of the jump conditions for the SLP~\cite{Kress2010}, use of
the representation
\begin{equation}\label{e:S}
  u(x) = S[\phi](x),\quad x\in\Omega,
\end{equation}
for the eigenfunction $u$, the Steklov boundary condition in
equation~\eqref{e:steklovM} gives rise to the generalized eigenvalue
problem
\begin{equation}\label{e:gen_eig_sing}
(\frac{1}{2}\mathcal{I}+\mathcal{T})\left[ \phi\right]= {\sigma}\, \mathcal{S}[\phi] \quad\mbox{for}\quad x \in \partial \Omega.
\end{equation}
Unfortunately, however, the single layer operator $\mathcal{S}$ on the
right side of this equation is not always invertible. In order to
avoid singular right-hand sides and the associated potential
sensitivity to round-off errors, in what follows we utilize the Kress
potential
\begin{equation}\label{e:S0}
  u(x) = S_0[\phi](x) = \int \limits_{\partial \Omega} G(x,y)
  \left(\phi(y)-\overline{\phi}\right) dS(y) + \overline{\phi},\quad x \in\Omega
\end{equation}
(where $\overline{\phi}$ denotes the average of $\phi$ over
$\partial\Omega$), which leads to the modified eigenvalue
equation~\cite{akhmetgaliyev2016thesis}
\begin{equation}\label{e:gen_eig_reg}
(\frac{1}{2}\mathcal{I}+\mathcal{T})\left[ \phi - \overline{\phi}\right]= {\sigma}\left(\mathcal{S}[\phi-\overline{\phi}]+ \overline{\phi} \right)\quad\mbox{for}\quad x \in \partial \Omega.
\end{equation}
The right-hand operator in this equation is
invertible~\cite[Thm. 7.41]{Kress2010}, as desired.  For either
formulation, the evaluation of a given eigenfunction $u$ requires
evaluation of the SLP, in accordance with either~\eqref{e:S}
or~\eqref{e:S0}, for the solution $\phi$ of the corresponding
generalized eigenvalue problem~\eqref{e:gen_eig_sing}
or~\eqref{e:gen_eig_reg}, respectively, at all required points
$x\in \Omega$.

\begin{rem}
{Note that for a given harmonic function $u$ in $\Omega$, $\phi$ in~\eqref{e:gen_eig_sing} and that in~\eqref{e:gen_eig_reg} are \emph{not} the same. }
\end{rem}

\subsection{Fourier expansion and exponential decay}

In terms of a given $2\pi$-periodic parametrization $C(t)$ of
$\partial \Omega$, the Steklov eigenfunction $u$ corresponding to a
given solution $(\phi,{\sigma})$ of the regularized eigenvalue
problem~\eqref{e:gen_eig_reg}, which is given by the single layer
expression~\eqref{e:S0}, can be expressed, for a given point
$x = (x_1,x_2)\in\Omega$,
\begin{equation}\label{u_param}
  u(x_1,x_2) = 
  \overline{\phi}{+}\frac{1}{{4}\pi} \int \limits_0^{2\pi} 
  \log \left[ \left( x_1 - C_1(t) \right)^2 + \left( x_2 -C_2(t) \right)^2 \right] \left[\phi\left(C(t)\right) - \overline{\phi} \right] \left| \dot{C}(t) \right| dt,
\end{equation}
where $C(t)=(C_1(t),C_2(t))$ and where $\overline{\phi}$ denotes the
average of $\phi$ over the curve $\partial\Omega$.  Unfortunately, a
direct use of this expression does not capture important elements in
the eigenfunction within $\Omega$, such as the nodal sets, since, for
analytic domains, the eigenfunctions decay exponentially fast within
$\Omega$ as the frequency increases~\cite{PoShTo,GalkToth}.  In
regions where the actual values of the eigenfunction may be
significantly below machine precision the expression~\eqref{u_param}
must be inaccurate: this expression can only achieve the exponentially
small values via the cancellations that occur as the the solution
$\phi$ becomes more and more oscillatory. But such cancellations
cannot take place numerically below the level of machine precision. In
order to capture the decay explicitly within the numerical algorithm
we proceed in a manner related to the construction used
in~\cite{PoShTo}.

To accurately obtain the exponentially decaying values of the Steklov
eigenfunction we proceed as follows. We first consider the Fourier
expansion
\begin{equation}\label{e:four_exp}
\left[\phi\left(C(t)\right) - \overline{\phi}\, \right]\left|
  \dot{C}(t) \right| = \sum_{\substack{n\in \mathbb{Z}\\ n\ne 0}} A_n
e^{int}.
\end{equation}
of the product
$\left[\phi(C(t))-\overline{\phi}\,\right]\left| \dot{C}(t) \right|$;
note that, as is easily checked, the $n=0$ term in the Fourier
expansion~\eqref{e:four_exp} is indeed equal to zero.  Inserting this
expansion in~\eqref{u_param} we obtain
\begin{equation*}
  u(x_1,x_2) = \overline{\phi}+\sum_{\substack{n\in \mathbb{Z}\\ n\ne 0}} A_n 
  B_n^0(x_1,x_2),\quad \mbox{where}
\end{equation*}
\[
B_n^0(x_1,x_2) = 
-\frac{1}{{4}\pi} \int \limits_0^{2\pi} 
\log \left[ \left( x_1 - C_1(t) \right)^2 + \left( x_2 -C_2(t) \right)^2 \right] 
e^{int}dt.
\]
Then, assuming an analytic boundary, as is relevant in the context of
this paper, and further assuming, for simplicity, that $C(t)$ is in
fact an entire function of $t$ (as are, for example, all
parametrizations $C(t)$ given by vector Fourier series containing
finitely many terms), we introduce, for $x=(x_1,x_2)\in\Omega$, the
quantities
\[
  {\lambda}(x) = \sup \left \{ s\geq 0 : x\neq C(t+ ir)\mbox{ for all
      $r$ with $|r|\leq s$ and for all $t \in [0,2\pi]$}\right\}
\]
and
\begin{equation}
\label{e:log-comp}
  B_n(x_1,x_2,s) = -\frac{1}{{4}\pi} \int \limits_0^{2\pi} \log \left[
    \left( x_1 - C_1(t+is\sgn(ns) ) \right)^2 + \left( x_2 -C_2(t+is\sgn(ns) ) \right)^2
  \right] e^{int}dt.
\end{equation}
Using Cauchy's Theorem for $x=(x_1,x_2)\in\Omega$ and any
$s\in\mathbb{R}$ satisfying $|s| \leq {\lambda}(x)$, we obtain
\begin{equation}\label{e:cauchy_coeffs}
B_n^0(x_1,x_2)=e^{-|ns|} B_n(x_1,x_2,s),
\end{equation}
and, thus, letting $s = \alpha{\lambda}(x)$ for any $\alpha\in\mathbb{R}$
satisfying $|\alpha|\leq 1$, the eigenfunction $u$ is given by
\begin{equation}\label{eig_sigma}
  u(x_1,x_2) = \overline{\phi}+\sum_{\substack{n\in \mathbb{Z}\\ n\ne
      0}} A_n e^{-|n\alpha| {\lambda}(x_1,x_2)} B_n(x_1,x_2,\alpha {\lambda}(x_1,x_2))
\end{equation}

\begin{lemma} 
\label{l:Bn}
There is $C>0$ such that for all $n>0$, 
$$
|B_n(x_1,x_2,\lambda(x_1,x_2))|\leq \frac{C}{1+|n|}.
$$
Moreover, there is $c>0$ and a sequence $\{n_k\}_{k=1}^\infty$ with $|n_k|\to \infty$ such that
\begin{equation}
\label{e:asympForm}
 |B_{n_k} ( x_1,x_2,{\lambda}(x_1,x_2) ) |\geq  \frac{c}{n_k}.
\end{equation}
\end{lemma}
\noindent 
A proof of Lemma~\ref{l:Bn} is given in Appendix~\ref{a:lemma}. It
follows from Lemma~\ref{l:Bn} that equation~\eqref{e:cauchy_coeffs}
optimally captures the exponential decay of the $B_n$ terms as
$\sigma\to\infty$. Note that this setup does not capture the
exponential decay of the coefficients $A_n$ below machine precision
away from $|n|\sim\sigma$, and, therefore, the accuracy of the
resulting interior eigenfunction reconstructions does not exceed that
accuracy level.  But the function $\lambda(x_1,x_2)$ does capture
the exponential decay and the geometrical character of the
eigenfunction as long as the (spatially constant) coefficients $A_n$
for low $n$ remain above machine precision. 

\begin{figure}[!h]		
  \centering		 
  \includegraphics[width=.65\textwidth]{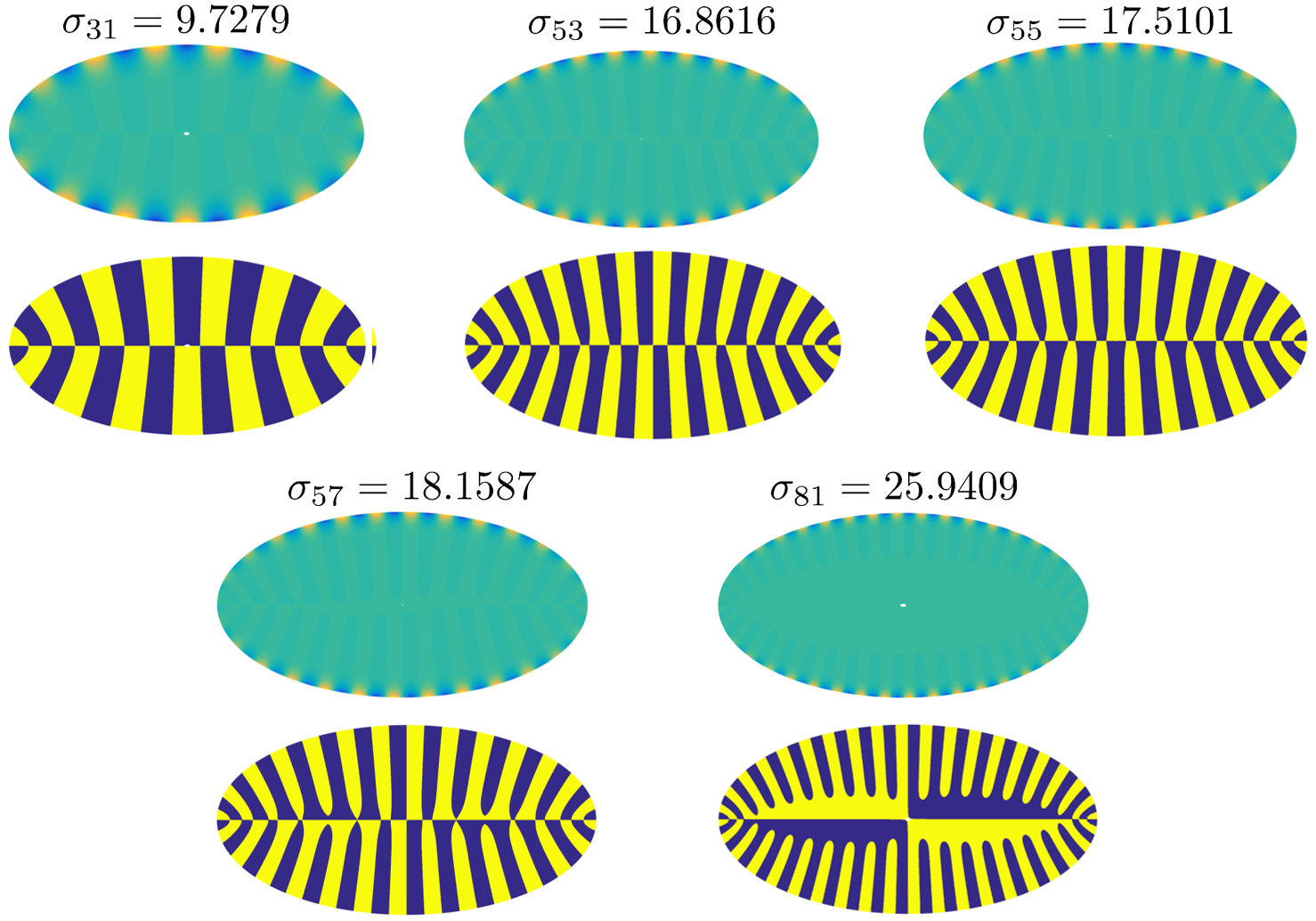}		 
  \caption{Density-plots (first and third rows) and fixed-sign sets
    (second and forth rows) for Steklov eigenfunctions over the
    elliptical domain~\eqref{ellipt_domain}. The eigenfunctions of
    orders $57$ and $81$ demonstrate the onset of the asymptotic
    character. In particular, regions of asymptotically fixed size open
    up. 
     \label{ellipses_31-81}}
\end{figure}
\begin{figure}[!h]
  \centering
  \includegraphics[width=.65\textwidth]{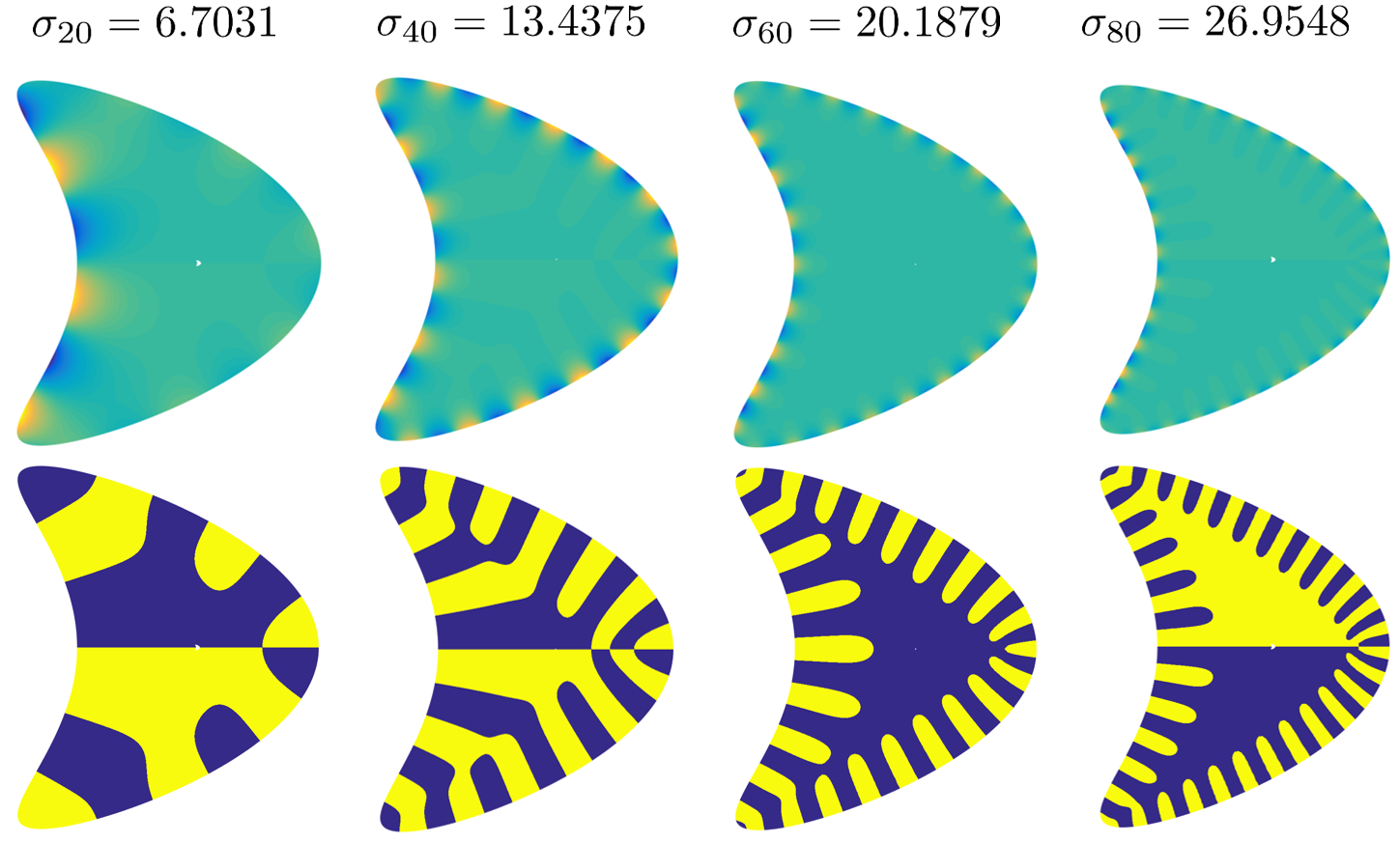}
  \caption{Density-plots (first row) and fixed-sign sets (second
    rows) for Steklov eigenfunctions over the kite-shaped
    domain~\eqref{kite_domain}. \label{kites_20-60}
  }
\end{figure}
For general curves $C(t)$ no closed form expressions exist for the
function ${\lambda}(x)$, and a numerical algorithm must be used for the
evaluation of this quantity, as part of a numerical implementation of
the eigenfunction expression~\eqref{eig_sigma}. In our implementation
the function ${\lambda}$ was evaluated via an application of Newton's
method to the nonlinear equation
\[
h(z)=(x_1-C_1(z))^2+(x_2-C_2(z))^2=0.
\]
Explicit expressions can be obtained for circles and ellipses,
however:
\begin{enumerate}
\item For a circle of radius 1: 
\[{\lambda}(x_1,x_2) = -\log \left( \sqrt{x_1^2 + x_2^2}\right) .\]
\item For an ellipse of semiaxes $a>b$: 
  \begin{equation}\label{sigma_ell}
  {\lambda}(x_1,x_2) = \mathrm{arcosh}\left( \frac{a}{\sqrt{a^2-b^2} }\right) - \mathrm{Re}\left\{ \mathrm{arcosh} \left( \frac{x_1+ix_2}{\sqrt{a^2-b^2}}\right) \right\}.
\end{equation}
\end{enumerate}
The derivation of the expression~\eqref{sigma_ell} is outlined in
Appendix~\ref{app}.

\subsection{Exponential decay and verification of Cauchy's
  theorem\label{verification}}
Tables~\ref{table:ellipse} and~\ref{table:kite} demonstrate the
validity of equation~\eqref{e:cauchy_coeffs} (since in both cases the
results in the second and third columns closely agree with each other
for $n\leq 50$), as well as the exponential decay of the exact
coefficients $B_n^0$---as born by the results in the third column of
these tables.  The disagreement observed for $n>50$ is caused by the
lack of precision of the results in the second column beyond machine
accuracy, a problem that is eliminated in the third column via an
application of the relation~\eqref{e:cauchy_coeffs}.
\begin{table}[h!]
\begin{center} 
  \begin{tabular}{|c|c|c|c|c|}
    \hline
    $n$   &  $|B_n^0(x_1,x_2)|$  & $|e^{-n0.8{\lambda}}B_n(x_1,x_2,0.8{\lambda})|$ &  Absolute $B_n^0$ error & Relative $B_n^0$  error\\ 
    \hline
    1   & 5.62e-03 & 5.62e-03 & 3.82e-16 & 6.79e-14 \\
    10  & 2.29e-06 & 2.29e-06 & 4.39e-17 & 1.91e-11 \\
    50  & 6.40e-16 & 6.57e-16 & 3.85e-17 & 5.86e-02 \\
    100 & 3.05e-17 & 1.30e-28 & 3.05e-17 & 2.35e+11 \\
    150 & 1.33e-16 & 5.95e-41 & 1.33e-16 & 2.23e+24 \\
    200 & 2.65e-16 & 6.58e-53 & 2.65e-16 & 4.02e+36 \\
    \hline
  \end{tabular}
\end{center}
\caption{Verification of the Cauchy-theorem-based identity~\eqref{e:cauchy_coeffs}  for the domain $\Omega$ bounded by the elliptical curve~\eqref{ellipt_domain} with $a=2$ and $b=1$. \label{table:ellipse}}
\end{table}
\begin{table}
\begin{center}
  \begin{tabular}{|c|c|c|c|c|}
    \hline
    $n$   &  $|B_n^0(x_1,x_2)|$  & $|e^{-n0.8{\lambda}}B_n(x_1,x_2,0.8{\lambda})|$ &  Absolute $B_n^0$ error & Relative $B_n^0$  error \\ 
    \hline
    1   & 5.83e-03 &  5.83e-03  & 4.25e-16 & 7.29e-14 \\
    10  & 5.97e-06 &  5.97e-06  & 7.18e-18 & 1.20e-12 \\
    50  & 2.33e-14 &  2.34e-14  & 3.32e-17 & 1.42e-03 \\
    100 & 1.14e-16 &  3.05e-25  & 1.14e-16 & 3.75e+08 \\
    150 & 1.27e-16 &  6.78e-36  & 1.27e-16 & 1.88e+19 \\
    200 & 2.42e-16 &  3.05e-45  & 2.42e-16 & 7.93e+28  \\
    \hline
  \end{tabular}
\end{center}
\caption{Same as Figure~\eqref{table:ellipse} but for the kite-shaped domain $\Omega$ bounded by the curve~\eqref{kite_domain}.\label{table:kite}}
\end{table}

\section{Numerical Results}\label{s:numer_res}
Figures~\ref{ellipses_31-81} and~\ref{kites_20-60} present density
plots and fixed-sign sets for Steklov eigenfunctions over domains
bounded by the elliptical and kite-shaped curves parametrized by the
vector functions
\begin{equation}\label{ellipt_domain}
C(t) = ((a\cos(t),b\sin(t))\quad (0\leq t <2\pi)
\end{equation}
with $a=2$ and $b=1$, and
\begin{equation}\label{kite_domain}
C(t) = (\cos(t) + 0.65  \cos (2t) - 0.65,1.5  \sin(t))\quad (0\leq t <2\pi),
\end{equation}
respectively. These figures demonstrate, in particular, domain-opening
and non-density of nodal sets as discussed in Remark~\ref{disclaimer}.
\\ \ \\

\noindent {\sc Acknowledgements.} Thanks to Agustin Fernandez Lado for
writing the code numerical Steklov-eigenfunction solver and for
providing the derivation presented in Appendix~\ref{app}. Thanks also to Jared Wunsch for suggesting part of the proof of Lemma~\ref{l:Bn} The authors
are grateful to the American Institute of Mathematics where this
research began. J.G. is grateful to the National Science Foundation
for support under the Mathematical Sciences Postdoctoral Research
Fellowship DMS-1502661 and under DMS-1900434. O.B.  gratefully acknowledges support by NSF,
AFOSR and DARPA through contracts DMS-1411876, FA9550-15-1-0043 and
HR00111720035, and the NSSEFF Vannevar Bush Fellowship under contract
number N00014-16-1-2808. \\ \ \\ 

\par

\appendix
\section{Function ${\lambda}(x)$ For an ellipse of semiaxes
  $a>b$\label{app}}
Let $\gamma=\sqrt{a^2-b^2}$ and $\mu_0 = \mathrm{arcosh}(a/\gamma)$. Using
elliptical coordinates with foci $(\pm \gamma,0)$ to represent the point
$x = (x_1,x_2)$, so that $x_1=\gamma\cosh(\mu)\cos(\tau)$ and
$x_2 =\gamma\sinh(\mu)\sin(\tau)$, and letting the boundary of the ellipse
be given by $C_1(t)=\gamma\cosh(\mu_0)\cos(t)$,
$C_2(t)=\gamma\sinh(\mu_0)\sin(t)$, in view of the relations
$x_1+ix_2 = \gamma\cosh(\mu+i\tau)$ and $C_1(t)+iC_2(t) = \gamma\cosh(\mu_0+it)$
we obtain
\begin{equation}
\begin{split}
(x_1-C_1(t+is))^2+&(x_2-C_2(t+is))^2 = \gamma^2 \left|\cosh ( \mu + i \tau )-\cosh(\mu_0+i(t+is)) \right|^2 \\
&= 4 \gamma^2 \left| \sinh \frac{\mu+\mu_0+i(\tau+(t+is))}{2} \right|^2 \left| \sinh \frac{\mu-\mu_0+i(\tau-(t+is))}{2} \right|^2.
\end{split}
\end{equation}
It follows that the left-hand side of this equation vanishes for some
value of $t$ if and only if either $s = (\mu_0 - \mu)$ or
$s = (\mu_0 + \mu)$. Thus, ${\lambda}(x)$ equals the smallest of these
two positive numbers, namely ${\lambda}(x) = (\mu_0 - \mu)$, which is
equivalent to the desired relation~\eqref{sigma_ell}.

\section{Proof of Lemma~\ref{l:Bn}}
\label{a:lemma}
First, let 
$$
h(z,x_1,x_2):=(x_1-C_1(z))^2+(x_2-C_2(z))^2. 
$$
Then, for $|\Im z|<\lambda(x_1,x_2)$, the expression
$$
\log h(z):=\int_0^z\frac{h'(s)}{h(s)}ds+\log h(0)
$$ 
defines the principal branch of $\log h(z)$---which is, then, an
analytic function in the strip $|\Im z|<\lambda$. On
$\pm \Im z=\lambda$, we define
$$
\log h(z):=\lim_{\e \to 0^+}h(z\mp i\e).
$$

\begin{lemma}
\label{l:jump}
Let $h(z)$ denote an analytic function defined on an open neighborhood
of the set $\{z:|\Im z|\leq \lambda\}$ which does not vanish for
$|\Im z|< \lambda$, but which vanishes to order $k$ at
$z_0=t_0+ i\lambda$.  Then,
$$
\lim_{\e_1\to 0^+}\Im\log h(z_0+\e_1)-\lim_{\e_2 \to 0^+}\Im h(z_0-\e_2)= k\pi .
$$
Similarly, if $h$ vanishes to order $k$ at $z_0=t_0-i\lambda$,
$$
\lim_{\e_1\to 0^+}\Im\log h(z_0+\e_1)-\lim_{\e_2 \to 0^+}\Im h(z_0-\e_2)=- k\pi .
$$
\end{lemma}
\begin{proof}
Note that for $\e>0$ small enough $\{h(z)=0\}\cap \{|z-z_0|<\e\}=z_0$. Therefore
$$
\log h(z_0+\e_1)-\log h(z_0-\e_2)=\int_{\Gamma}\frac{h'(z)}{h(z)}dz
$$
where $\Gamma$ is any contour starting at $z_0-\e_2$, ending at $z_0+\e_1$, and lying in
$$
\{\Im z\leq \lambda\}\cap B(z_0,\e).
$$
In particular, let 
$$
\Gamma_1=\{ z_0+\e_2e^{it}\mid t\in [\pi,2\pi]\},\qquad \Gamma_2:=\{ z_0 +(1-t)\e_2+t\e_1\}
$$
and $\Gamma=\Gamma_1\cup \Gamma_2$. Then, since 
$$
\frac{h'(z)}{h(z)}= \frac{k}{z-z_0}(1+O(|z-z_0|)), 
$$
$$
\log h(z_0+\e_1)-\log h(z_0-\e_2)= k\pi i +\log \e_1 - \log \e_2+O(|\e_1-\e_2|)+O(\e_2)
$$
Letting $\e_1$ and $\e_2$ tend to zero completes the proof for the
case $z_0=t_0+i\lambda$. The proof for $z_0=t_0-i\lambda$ follows by
substituting $z$ by $-z$.
\end{proof}

\begin{lemma}
\label{l:integrate}
Let $h(z,x_1,x_2)$ denote an analytic function on
$|\Im z|\leq \lambda$ which vanishes to order $k$ at
$z_0=t_0+ i\lambda$.  Then for $\chi\in C_c^\infty(S^1)$ supported in
a sufficiently small neighborhood of $t_0$, with $\chi\equiv 1$ near
$t_0$, we have
$$
\int_{S^1} \chi(t)\log h(t+i\lambda)e^{int}dt=-\frac{2\pi
  k}{|n|}e^{int_0} +O(n^{-2})\quad \mbox{for $n>0$}.
$$
Similarly if $ h$ vanishes to order $k$ at $z_0=t_0-i\lambda$, we have
$$
\int_{S^1} \chi(t)\log h(t-i\lambda)e^{int}dt=-\frac{2\pi
  k}{|n|}e^{int_0} +O(n^{-2})\quad \mbox{for $n<0$}.
$$
\end{lemma}
\begin{proof}
  We consider the first case, the second follows similarly. 

  Selecting $\chi(t)$ with sufficiently small support we ensure that,
  within the support of $\chi$, $h(t+i\lambda)$ vanishes only at
  $t=t_0$. We then have
\begin{equation}\label{e:chi_int_1}
\int \chi(t)\log [h(t+i\lambda)]e^{int}dt=\int\chi(t)\left( \log |h(t+i\lambda)|+i\Im \log[h(t+i\lambda))]\right)e^{int}dt
\end{equation}
and
\begin{equation}\label{e:chi_int_2}
\int \chi(t)\log |h(t+i\lambda)|e^{int}dt=\int\chi(t)\left( k\log |t-t_0|+\log |t-t_0|^{-k}|h(t+i\lambda)|\right)e^{int}dt.
\end{equation}
Since $|t-t_0|^{-k}|h(t+i\lambda)|$ is smooth and bounded away from
zero on the support of $\chi$, the second term in~\eqref{e:chi_int_2}
is $O(n^{-\infty})$.

Taking real parts in the asymptotic
formula~\cite[p. 381]{bender2013advanced} we obtain
\begin{equation}
\label{e:logInt}
\int_{-1}^1\log |t|e^{ixt}dt= -\frac{\pi}{|x|}+O(x^{-2}),\qquad x\to \infty.
\end{equation}
Then, using~\eqref{e:logInt} together with the fact that $\log 1=0$ we
may approximate the first term on the right-hand side
of~\eqref{e:chi_int_2} by
$$
\int\chi(t)\log |h(t+i\lambda)|e^{int}dt= -\pi ke^{int_0}\frac{1}{|n|}+O(n^{-2}),
$$

Let us now estimate the second term on the right-hand side
of~\eqref{e:chi_int_1}. We have
\begin{align*}
&\int\chi(t)i\Im \log[h(t+i\lambda))]e^{int}dt\\
&=\int_{0}^{t_0}i\chi(t)\Im \log [h(t+i\lambda))]e^{int}dt+\int_{t_0}^{2\pi}i\chi(t)\Im \log[h(t+i\lambda))]e^{int}dt\\
&=-n^{-1}\Big(\int_{0}^{t_0}\partial_t(\chi(t)\Im \log [h(t+i\lambda)])e^{int}dt+\int_{t_0}^{2\pi}\partial_t(\chi(t)\Im \log [h(+-i\lambda)])e^{int}dt\Big)\\
&\qquad -n^{-1}(e^{int_0}(\lim_{t\to t_0^+}\Im \log [h(t+i\lambda)])-\lim_{t\to t_0^-}\Im \log[h(t+i\lambda)])\\
&=-n^{-1}(e^{int_0}(\lim_{t\to t_0^+}\Im \log [h(t+i\lambda)])-\lim_{t\to t_0^-}\Im \log[h(t+i\lambda)])+O(n^{-2})\\
&=-k\pi n^{-1}e^{int_0}+O(|n|^{-2})
\end{align*}
where in the last equality Lemma~\ref{l:jump} was used.
\end{proof}

We may now complete the proof of Lemma~\ref{l:Bn}. Let
$0\leq t_1<t_2<\dots <t_M<2\pi$ denote the zeroes of $h(t+i\lambda)$
as a function of $t$, and let $k_j$ ($0\leq j\leq M$) denote the
vanishing order at $t=t_j$. Then, by Lemma~\ref{l:integrate}, for
$\chi_j$ supported close enough to $t_j$ with $\chi_j\equiv 1$ near
$t_j$, and $n>0$,
$$
\int \chi_j(t)\log h(t+i\lambda)e^{int}dt= -\frac{2\pi k_je^{int_j}}{|n|}+O(n^{-2}).
$$
By shrinking the support of $\chi_j$, we may assume that $\supp \chi_j\cap \chi_\ell=\emptyset$ for $\ell\neq j$. Then, since $\chi_j\equiv 1$ near $t_j$,  $(1-\sum_j \chi_j(t)))\log h(t+i\lambda)\in C^\infty(S^1)$ and hence
$$
\int (1-\sum_j \chi_j(t)))\log h(t+i\lambda)e^{int}dt=O(n^{-\infty}).
$$
Thus in view of equation~\eqref{e:log-comp} we obtain
$$
B_n(x_1,x_2,\lambda(x_1,x_2)) = \int \log h(t+i\lambda)e^{int}dt=-\frac{2\pi}{|n|}\sum_{j=1}^M k_je^{int_j}+O(n^{-2})
$$

Proceeding by contradiction, assume that
\begin{equation}
\label{e:toContradict}
\limsup_{n\to +\infty }n|B_n(x_1,x_2,\lambda(x_1,x_2))|=0.
\end{equation}
Then in particular, 
$$
\lim_{n\to +\infty} \sum_{j=1}^Mk_je^{int_j}=0.
$$
But we note that
\begin{align*}
\lim_{N\to \infty}\frac{1}{N}\sum_{n=0}^{N-1}\Big|\sum_{j=1}^Mk_je^{int_j}\Big|^2&
=\sum_{j=1}^Mk_j^2+\lim_{N\to \infty}\frac{1}{N}\sum_{j\neq \ell}\sum_{n=0}^{N-1}k_jk_\ell e^{in(t_j-t_\ell)}\\
&=\sum_{j=1}^Mk_j^2+\lim_{N\to \infty}\frac{1}{N}\sum_{j\neq \ell}k_jk_\ell\frac{1-e^{iN(t_j-t_\ell)}}{1-e^{i(t_j-t_\ell)}}=\sum_{j=1}^Mk_j^2>0.
\end{align*}
Recalling that
$$
\limsup_{N\to \infty}\frac{1}{N}\sum_{n=0}^{N-1}a_n\leq \limsup_{n\to \infty} a_n
$$
we obtain
$$
\lim_{n\to \infty} \sum_{j=1}^Mk_je^{int_j}\neq 0.
$$
which contradicts~\eqref{e:toContradict}.

If $h(t+i\lambda)$ does not vanish anywhere, then $h(t-i\lambda)$ vanishes at some $0\leq t_1<t_2<\dots <t_M<2\pi$ and we may repeat the argument this time considering 
$$
B_n(x_1,x_2,\lambda(x_1,x_2))=\int \log h(t-i\lambda) e^{int}dt,\qquad n<0.
$$
and taking the limit as $n\to -\infty$.

\bibliography{biblio}
\bibliographystyle{alpha}

\end{document}